\documentclass[a4paper,11pt]{amsart}
\usepackage{amssymb}
\usepackage{amsthm}
\usepackage[utf8]{inputenc}
\usepackage[T1]{fontenc}
\usepackage{lmodern}
\usepackage{graphicx}
\usepackage[a4paper]{geometry}
\usepackage{url}
\usepackage{float}
\usepackage[Rejne]{fncychap}
\usepackage{amsmath}
\usepackage[colorlinks=true,urlcolor=blue,linkcolor=blue,citecolor=red]{hyperref}
\newcommand{\R}{\mathbb{R}}
\newcommand{\hy}{\mathbb{H}}

\newcommand{\pr}{\mathbb{P}}

\newcommand{\eu}{\widetilde{Ein}_{1,n-1}}
\newcommand{\eeu}{\widetilde{Ein}_{1,n}}
\newcommand{\G}{\Gamma}


\title[Anosov representations and conformally flat spacetimes]{ANOSOV REPRESENTATIONS AS HOLONOMIES OF GLOBALLY HYPERBOLIC SPATIALLY COMPACT CONFORMALLY FLAT SPACETIMES}

\author{Rym SMAI} 

\date{}

\address{LMA, Avignon University, Campus Jean-Henri Fabre, 301, rue Baruch de Spinoza, BP 21239, F-84 916 Avignon Cedex 9, France.}

\email{rym.smai@univ-avignon.fr}

\begin{document} 

\maketitle

\begin{abstract}
Anosov representations were introduced by F. Labourie \cite{labourie} for fundamental groups of closed negatively curved surfaces, and generalized by O. Guichard and A. Wienhard \cite{GW} to representations of arbitrary Gromov hyperbolic groups into real semisimple Lie groups. In this paper, we focus on Anosov representations into the identity component $O_0(2,n)$ of $O(2,n)$ for $n \geq 2$. Our main result is that any Anosov representation with negative limit set as defined in \cite{AnoCCH} is the holonomy group of a spatially compact, globally hyperbolic maximal (abbrev. CGHM) conformally flat spacetime. The proof of the spatial compactness needs a particular care. The key idea is to notice that for any spacetime $M$, the space of lightlike geodesics of $M$ is homeomorphic to the unit tangent bundle of a Cauchy hypersurface of $M$. For this purpose, we introduce \emph{the space of causal geodesics} containing timelike and lightlike geodesics of anti-de Sitter space and lightlike geodesics of its conformal boundary: the Einstein spacetime.  The spatial compactness is a consequence of the following theorem : Any Anosov representation acts properly discontinuously by isometries on the set of causal geodesics avoiding the limit set; besides, this action is cocompact. It is stated in a general setting by O. Guichard, F. Kassel and A. Wienhard in \cite{guichard2015tameness}. We can see this last result as a Lorentzian analogue of the action of convex cocompact Kleinian group on the complementary of the limit set in $\overline{\hy^n}$. Lastly, we show that the conformally flat spacetime in our main result is the union of two conformal copies of a strongly causal AdS-spacetime with boundary which contains - when the limit set is not a topological $(n-1)$-sphere - a globally hyperbolic region having the properties of a \emph{black hole} as defined in \cite{Baados1993GeometryOT}, \cite{Baados1992BlackHI}, \cite{Barbot2005CausalPO}.
\end{abstract} 

\theoremstyle{plain}

\newtheorem{theorem}{Theorem}[section]

\newtheorem{proposition}{Proposition}[section]

\newtheorem{corollary}{Corollary}[section]

\newtheorem{lemma}{Lemma}[section]

\theoremstyle{definition}

\newtheorem{definition}{Definition}[section]

\theoremstyle{remark}

\newtheorem{tech lemma}{Fact}[section]

\newtheorem{criterion}{Criterion}

\newtheorem{remark}{Remark}[section]

\newtheorem{example}{Example}[section]

\section{\textbf{Introduction}}

In the whole paper, we fix an integer $n \geq 2$. Let $O(2,n)$ be the orthogonal group of a non-degenerate symmetric bilinear form $<.,.>$ of signature $(2,n)$ on $\mathbb{R}^{n+2}$ and let $O_0(2,n)$ be its identity component. We study the link between Anosov representations of hyperbolic groups into $O_0(2,n)$ and conformally flat spacetimes. Roughly speaking, we show how to construct a spatially compact globally hyperbolic maximal conformally flat spacetime starting from an Anosov representation. Except in a particular case, we prove that this spacetime contains a region which can be interpreted as a black hole.

\subsection{Conformally flat spacetimes}
 
A Lorentzian manifold is conformally flat if each point has a neighborhood that can be mapped into the flat Minkowski space $\R^{1,n}$ by a conformal isometry. In a Lorentzian manifold $M$, since the signature of the metric is $(1,n)$, we distinguish three type of vectors : \emph{spacelike}, \emph{lightlike} and \emph{timelike}, which correspond to vectors of \emph{positive}, \emph{null} and \emph{negative} norm respectively. A curve in $M$ is spacelike, lightlike or timelike according to the type of the tangent vectors. Non-spacelike curves are said to be \emph{causal}. A Lorentzian manifold is called spacetime if it is oriented an admits a time orientation, i.e., an orientation of every causal curve. In a spacetime, the \emph{lightlike cone} of a point $p$ is the union of all the lightlike curves going through $p$. We denote it by $\mathcal{L}(p)$. A spacetime is \emph{strongly causal} if every point admits a neighborhood $U$ such that any causal curve with extremities in $U$ is contained in $U$. The \emph{causality} of spacetimes is developed in Section \ref{1.1}. In this paper, we deal with two conformally flat spacetimes models : \emph{anti-de Sitter space} and \emph{Einstein universe}. Detailed descriptions of these spaces are given in Section \ref{1.2} and Section \ref{1.3} respectively.

The anti-de Sitter space $AdS_{1,n}$ is the Lorentzian analogue of the hyperbolic space $\hy^{n+1}$. It is a spacetime, of constant sectional curvature $-1$. The group $O_0(2,n)$ is precisely the group of orientation and time-orientation preserving isometries of $AdS_{1,n}$ (see Section \ref{1.4}).

The anti-de Sitter space $AdS_{1,n}$ admits a conformal boundary called the Einstein universe and denoted by $Ein_{1,n-1}$, which is the Lorentzian analogue of the conformal boundary $\mathbb{S}^n$ of the hyperbolic space $\mathbb{H}^n$. The Einstein universe is a conformal Lorentzian spacetime; it means that $Ein_{1,n}$ is an oriented manifold equipped with a conformal class of Lorentzian metrics. Since the type of vectors is invariant under conformal changes of metrics, we also have a causality notion in $Ein_{1,n}$. A subset $\Lambda$ of the Einstein space $Ein_{1,n-1}$ is called \emph{acausal} if any pair of distinct points in $\Lambda$ are the extremities of a spacelike geodesic in $AdS_{1,n}$. The group $O(2,n)$ acts by conformal isometries on $Ein_{1,n-1}$. In particular, the isometries of $O_0(2,n)$ are orientation and time-orientation preserving. Since $Ein_{1,n-1}$ is conformally flat, it follows that $(O_0(2,n), Ein_{1,n-1})$-spacetimes are conformally flat. By a Lorentzian version of Liouville's theorem (see Theorem \ref{Liouville}), when $n \geq 3$, the group $O(2,n)$ coincides with the group of conformal isometries of $Ein_{1,n-1}$. Besides, the elements of $O_0(2,n)$ are exactly those which are orientation and time-orientation preserving (see Section~\ref{1.4}). Therefore, conformally flat spacetimes of dimension $n \geq 3$ are exactly those which are locally modeled on the Einstein spacetime, i.e. equipped with a $(O_0(2,n),Ein_{1,n-1})$-structure (see Section \ref{1.5}). Similarly to the Riemannian case where $\hy^n$ embeds conformally in the conformal sphere $\mathbb{S}^n$, the anti-de Sitter space $AdS_{1,n}$ embeds conformally in $Ein_{1,n}$. Actually, we can say more: $Ein_{1,n}$ is the union of two conformal copies of $AdS_{1,n}$ glued along their conformal boundary $Ein_{1,n-1}$.\\

\paragraph{\textit{Spatially compact globally hyperbolic spacetimes.}} In a spacetime $M$, a causal curve $\gamma$ is said to be inextensible if there is no curve in $M$ that extend $\gamma$. A spacetime $M$ is \emph{globally hyperbolic} (abbrev. GH) if there is a Riemannian hypersurface $S$ in $M$ such that every inextensible causal curve intersects $S$ in a unique point. In this case, $S$ is called a \emph{Cauchy-hypersurface} of $M$. Besides, if $S$ is compact, the spacetime $M$ is said to be \emph{spatially compact} (abbrev. CGH). The CGH spacetimes are studied in Section \ref{1.1}.\\

\paragraph{\textit{Maximality.}} Let $M$ and $N$ two globally hyperbolic conformally flat spacetimes. A \emph{Cauchy-embedding} of $M$ into $N$ is a conformal embedding which send every Cauchy hypersurface of $M$ on a Cauchy hypersurface of $N$. A globally hyperbolic conformally flat spacetime $M$ is said to be \emph{maximal} if every Cauchy embedding into any other spacetime is onto (see Section \ref{1.5}).

\subsection{Anosov representations}

Anosov representations were introduced by F. Labourie \cite{labourie} for fundamental groups of closed negatively curved surfaces and extended by O. Guichard and A. Wienhard to arbitrary Gromov hyperbolic groups \cite{GW}. In fact, it turns out that any finitely generated group admitting an Anosov representation is Gromov hyperbolic. This is a consequence of a general result obtained by M. Kapovitch, B. Leeb and J. Porti on quasi-isometric embeddings of metric spaces satisfying a certain condition (see \cite{Kapovich2018}, Theorem $1.4$). Later, it has been proved again in a more elementary way, by J. Bochi, R. Potri and A. Sambarino (see \cite{Sambarino}, sections $3$ and $4$). Throughout this paper, we adopt the definition used in \cite{guichard2015tameness} (taken from \cite{Gueritaud2017Anosov}) which assume that the represented group is Gromov hyperbolic (from now on, we simply say hyperbolic). Besides, we focus only on Anosov representations into $O_0(2,n)$. Lastly, let us point out that Anosov representations can be thought, in many respects, as a generalization to a higher rank setting, of convex cocompact representations into rank one simple groups (see \cite{GW}, \cite{AnoCCH}).\\

Roughly speaking, $P_1$-Anosov representations of hyperbolic groups into $O_0(2,n)$ are group representations preserving some dynamical properties. The definition involves, on one hand, the dynamic under the action of a hyperbolic group on its Gromov boundary; and on the other hand, the dynamic under the action of $O_0(2,n)$ on the Einstein universe $Ein_{1,n-1}$. A hyperbolic group $\Gamma$ is a \emph{convergence group} for its action on its Gromov boundary $\partial \Gamma$ i.e. for any infinite sequence of distinct elements $\gamma_n$ in $\Gamma$, there exist a subsequence $\{\gamma_{n_k}\}_k$, an \emph{attracting} point $\xi_+$ and a \emph{repelling} one $\xi_-$ in $\partial \Gamma$ such that the maps $\{\gamma_{n_k}\}_k$ converge uniformly on compact subsets of $\partial \Gamma \backslash \{ \xi_- \}$ to the constant map sending $\partial \Gamma \backslash \{ \xi_- \}$ to $\xi_+$. This dynamical property of hyperbolic groups is well studied in \cite{Kapovich2002Boundaries}. In the Einstein universe, we observe an almost similar dynamic for the action of sequences of $O_0(2,n)$ satisfying the property of being \emph{$P_1$-divergent} (see Section \ref{2.1} for the definition). For every $P_1$-divergent sequence of elements $g_n$ in $O_0(2,n)$, there exist a subsequence $\{g_{n_k}\}_k$, an \emph{attracting} point $x_+$ and a \emph{repelling} one $x_-$ in $Ein_{1,n-1}$ such that the maps $\{g_{n_k}\}_k$ converge uniformly on compact subsets of the complementary in $Ein_{1,n-1}$ of the lightlike cone $\mathcal{L}(x_-)$ of $x_-$ to the constant map sending $Ein_{1,n-1} \backslash \mathcal{L}(x_-)$ to $x_+$. The limit set of a subgroup $G$ of $O_0(2,n)$ is the set of all attracting and repelling points of $P_1$-divergent sequences of $G$. This dynamic is described in details in Section \ref{2.1}. Now, we can give a precise definition of $P_1$-Anosov representations. Let $\Gamma$ be a hyperbolic group. We denote by $\partial \Gamma$ the Gromov boundary of $\Gamma$.

A representation $\rho : \G \longrightarrow O_0(2,n)$ is $P_1$-Anosov if
\begin{enumerate}
\item every sequence of pairwise distinct elements in $\rho(\G)$ is $P_1$-divergent;
\item there exist a continuous, $\rho$-equivariant boundary map $\xi : \partial \G \longrightarrow Ein_{1,n-1}$ which is
\begin{enumerate}
\item \emph{transverse}, meaning that every pair of distinct points $(\eta, \eta')$ in $\partial \G$ is sent to the pair $(\xi(\eta),\xi(\eta'))$ in $Ein_{1,n-1} \times Ein_{1,n-1}$ such that $\xi(\eta)$ is not contained in the lightlike cone of $\xi(\eta')$;
\item \emph{dynamics-preserving}, meaning that if $\eta$ is the attracting fixed point of some element $\gamma \in \G$ in $\partial \G$, then $\xi(\eta)$ is an attracting fixed point of $\rho(\gamma)$ in $Ein_{1,n-1}$. 
\end{enumerate}
\end{enumerate} 

In this case, the limit set of $\rho(\Gamma)$ coincides with $\xi(\partial \Gamma)$ (see Section \ref{2.2}). 

A subset $\Lambda$ of $Ein_{1,n-1}$ is \emph{negative} if all inner products $<.,.>_{2,n}$ of distinct points of $\Lambda$ are negative. It turns out that negative subsets of $Ein_{1,n-1}$ are the acausal ones (see Section \ref{1.3}).

Now, we can state our main result :

\begin{theorem}\label{main result}
Let $n \geq 2$. A $P_1$-Anosov representation of hyperbolic group into $O_0(2,n)$ with negative limit set is the holonomy group of a spatially compact globally hyperbolic maximal conformally flat spacetime.
\end{theorem}
				
Given a $P_1$-Anosov representation with negative limit set, we construct an invariant open domain $\Omega$ of $Ein_{1,n}$ on which the image of the representation acts properly discontinuously. We prove that the quotient manifold $M$ is a conformally flat CGHM spacetime. Since $Ein_{1,n}$ is the union of two copies of $AdS_{1,n}$ glued along their conformal boundary $Ein_{1,n-1}$, the open domain $\Omega \subset \eeu$ meets each of these copies in an $AdS$-regular domain $E_i$ ($i=1,2$). The quotient manifolds $M_i = \rho(\Gamma) \backslash E_i$ are strongly causal spacetimes which embed in $M$. When the limit set is not a topological $(n-1)$-sphere, we observe in $M_i$ a phenomenon which can be interpreted as \emph{a black hole} (see \cite{Baados1993GeometryOT}, \cite{Baados1992BlackHI}, \cite{Barbot2005CausalPO}) and that we describe in Section \ref{4.3}.\\

Lastly, we introduce in Section \ref{4.1} \emph{the space of causal geodesics} which is a central element in the proof of spatial compactness in Theorem \ref{main result}.

\subsection{The space of causal geodesics}

In Section \ref{1.2} and in Section \ref{1.3}, we describe geodesics of anti-de Sitter space and Einstein universe respectively. Both are obtained by intersection with $2$-planes of $\R^{n+2}$. Timelike geodesics of $AdS_{1,n}$ correspond to $2$-planes where the restriction of $<.,.>$ is negative definite. Therefore, the space $\mathcal{T}^{2n}$ of timelike geodesics of $AdS_{1,n}$ admits a realization as the open subset
\[X = \{P \in Gr_{2}(\R^{n+2})\ :\ <x,x>\ < 0,\ \forall x \in P \}\]
in the Grassmannian space $Gr_{2}(\R^{n+2})$.
Its closure
\[\bar{X} = \{P \in Gr_{2}(\R^{n+2})\ :\ <x,x>\ \leq 0,\ \forall x \in P\}\]
contains besides negative-definite $2$-planes,
\begin{itemize}
\item \emph{isotropic $2$-planes}, i.e. $2$-planes such that the restriction of $<.,.>$ is degenerate, corresponding to \emph{lightlike geodesics of $AdS_{1,n}$};
\item\emph{totally isotropic $2$-planes} corresponding to \emph{lightlike geodesics of $Ein_{1,n-1}$}.
\end{itemize}
The space $\bar{X}$ is identified with the space containing timelike geodesics of $AdS_{1,n}$, lightlike geodesics of $AdS_{1,n}$ and lightlike geodesics of $Ein_{1,n-1}$, called \emph{the space of causal geodesics}. We denote it by $\mathcal{C}$. In section \ref{3}, we prove the following fact.

\begin{proposition}
The topological boundary of $\bar{X}$ in $Gr_{2}(\R^{2,n})$ is a topological manifold homeomorphic to the space $\mathcal{P}(Ein_{1,n})$ of lightlike geodesics of $Ein_{1,n}$.
\end{proposition}

The group $O_0(2,n)$ acts transitively on $X$. The isotropy group of an element $P$ in $X$ is equal to $SO(2) \times SO(n)$ which is a maximal compact subgroup of $O_0(2,n)$. Therefore, the manifold $X$ is diffeomorphic to the Riemannian symmetric space  $O_0(2,n) / SO(2) \times SO(n)$ and thus, so do $\mathcal{T}^{2n}$. Any discrete subgroup $G$ of $O(2,n)$ acts properly discontinuously by isometries on the Riemannian symmetric space $\mathcal{T}^{2n}$ of $O_0(2,n)$. The quotient space $G \backslash \mathcal{T}^{2n}$ is a locally symmetric space which is non-compact except if $G$ is a uniform lattice in $O(2,n)$. In \cite{guichard2015tameness}, Section $4$, O. Guichard, F. Kassel and A. Wienhard give a compactification of $G \backslash \mathcal{T}^{2n}$ when $G$ is the image of an Anosov representation and prove the following result.

\begin{theorem}\label{GKW0}
Let $\Gamma$ be a hyperbolic group and $\rho : \Gamma \longrightarrow O_0(2,n)$ a $P_1$-Anosov representation with boundary map $\xi : \partial \Gamma \longrightarrow Ein_{1,n-1}$. Let
\begin{align*}
U &= \{\varphi \in \mathcal{C}\ :\ \varphi \cap \xi(\partial \Gamma) = \emptyset\}
\end{align*}
be the subspace of causal geodesics that avoid the limit set $\xi(\partial \Gamma)$ of $\rho(\Gamma)$.\\
Then, the action of $\rho(\Gamma)$ on $U$ is properly discontinuous and cocompact. The set $U$ contains the Riemannian symmetric space $\mathcal{T}^{2n}$ and $\rho(\Gamma) \backslash U$ is a compactification of $\rho(\Gamma) \backslash \mathcal{T}^{2n}$.
\end{theorem}

In \cite{guichard2015tameness}, this theorem is stated in a general setting and in an algebraic way while we state it here in the setting of Lorentzian geometry. Moreover, we give a geometrical proof of it in section \ref{3}. Our theorem \ref{main result} is a consequence of this result.

\subsection{Organization of the paper}
In Section \ref{1} we recall some basic definitions in Lorentzian geometry. In particular we give descriptions of anti-de Sitter and Einstein spaces. In Section \ref{2} we recall the notions of limit set and Anosov representations of hyperbolic groups into $O_0(2,n)$. In Section \ref{3} we introduce the space of causal geodesic and give a geometrical proof of Theorem \ref{GKW0}. We devote Section \ref{4} to prove our main theorem \ref{main result}. In particular, we show that Theorem \ref{main result} gives us a family of examples of spacetimes with black holes.

\section{Preliminaries} \label{1}

\subsection{The causal structure of spacetimes} \label{1.1}

For the convenience of the reader, we recall some basic definitions in Lorentzian geometry. In the whole paper, we denote by $(n,p,r)$ the signature of a quadratic form $q$ where $n,p,r$ are respectively the numbers of negative, positive and zero coefficients in the polar expression of the quadratic form $q$. When $q$ is non-degenerate, we just write $(n,p)$.\\

\paragraph{\textit{Spacetimes.}} A Lorentzian $(n+1)$-manifold is a smooth $(n+1)$-manifold $M$ equipped with a nondegenerate symmetric $2$-form $g$ of signature $(1,n)$ called \emph{Lorentzian metric}. A non-zero tangent vector $v$ is \emph{timelike} (resp. \emph{lightlike}, \emph{spacelike}) if $g(v,v)$ is negative (resp. null, positive). We say that $v$ is \emph{causal} if it is non-spacelike. In each tangent bundle, the cone of causal vectors has two connected components. The Lorentzian manifold $M$ is \emph{time-orientable} if it is possible to make a continuous choice, in each tangent bundle, of one of them. A causal tangent vector is said to be \emph{future} if it is in the chosen one and \emph{past} otherwise. More precisely, a time orientation is given by a timelike vector field $X$. A tangent vector $v$ in $T_pM$ is future-oriented if $g_p(v,X(p))<0$ and past-oriented if $g_p(v,X(p)) > 0$. Remark that up to a double cover, $M$ is always time-orientable.

\begin{definition}
A spacetime is a connected, oriented and time-oriented Lorentzian \mbox{manifold.} 
\end{definition}

The basic examples of spacetimes are \emph{Minkowski, de-Sitter} and \emph{anti-de Sitter spaces} which are the spacetimes of constant sectional curvature respectively equal to $0$, $1$, and $-1$. They are the Lorentzian analogues of the Riemannian manifolds of constant sectional curvature : the \emph{euclidean space, the sphere} and \emph{the hyperbolic space} respectively. In this paper, we will focus, in particular, on the \emph{anti-de Sitter space} (see Section \ref{1.2}).

A \emph{differential causal curve} (resp. \emph{timelike, lightlike, spacelike}) of a spacetime $M$ is a curve $\mathcal{C}^1$ on $M$ such that at every point, the tangent vector to the curve is \emph{causal} (resp \emph{timelike, lightlike, spacelike}). As for Riemannian manifold, it is possible to define the Levi-Civita connection on $M$ which is the unique torsion-free connection on the tangent bundle of $M$ preserving its Lorentzian metric. Also, a geodesic of $M$ is a curve such that parallel transport along the curve preserves tangent vectors to the curve. However, geodesics are not considered as a minimizing curves anymore since a Lorentzian metric does not provide a distance on the spacetime. Since the type of tangent vectors to a geodesic is the same along the curve, the type of a geodesic of a spacetime is always well-defined. A lightlike geodesic is called a \emph{photon}. A causal curve is said to be \emph{future-oriented} (resp. \emph{past-oriented}) if all tangent vectors to the curve are future-oriented (resp. \emph{past-oriented}). It is possible to generalize the definition of future (resp.past) causal curve to piecewise differential curves. An important notion in the study of spacetimes is \emph{the causality}. By \emph{causality} we refer to the general question on which points in a spacetime can be joined by causal curves; on the relativity point of view, which events can influence a given event. This motivates the following definitions.

The \emph{causal future} $J^+_U(A)$ (resp. \emph{chronological future} $I^+_U(A)$) of a subset $A$ of spacetime $M$ relatively to an open subset $U$ containing $A$ is the set of future-ends of piecewise-smooth \emph{causal} curves (resp. \emph{timelike}) starting from a point of $A$ contained in $U$. Similarly, we define the \emph{causal past} $J^-_U(A)$ (resp. \emph{chronological past} $I^-_U(A)$) of a subset $A$ of spacetime $M$ relatively to an open subset $U$ containing $A$ by remplacing future by \emph{past} in the definition. If $U = M$, we just write $J^{\pm}(A)$ (resp. $I^{\pm}(A)$) instead of $J^{\pm}_M(A)$ (resp. $I^{\pm}_M(A)$). A spacetime $M$ could contain subsets where no point is causally related to another. A subset $A$ of $M$ is \emph{achronal} (resp. \emph{acausal}) if no timelike (resp. causal) curve meets $A$ more than once.\\

\paragraph{\textit{Conformal spacetimes.}} A conformal Lorentzian manifold is a smooth manifold equipped, no longer with a single Lorentzian metric, but with a conformal class of Lorentzian metrics. Notice that conformal changes of metrics does not change the type of tangent vectors since in each tangent space, the metric is multiplied by a positive scalar. However, geodesics are not preserved by conformal changes of metrics except lightlike geodesics.

\begin{theorem}
Let $(M,g)$ a pseudo-Riemannian manifold. Then, lightlike geodesics are the same, up to parametrization, for all metrics conformally equivalent to $g$.
\end{theorem}

A proof of this theorem is given in \cite{francesarticle}. A conformal spacetime is a connected, oriented, time-oriented conformal Lorentzian manifold. An important example of conformal spacetime is the \emph{Einstein universe} which is the Lorentzian analogue of the conformal sphere in Riemannian geometry (see Section \ref{1.3}).\\

\paragraph{\textit{Globally hyperbolic spacetimes.}} An important causal property of spacetimes is the global hyperbolicity.

\begin{definition}
A spacetime $M$ is globally hyperbolic (abbrev. GH) if
\begin{enumerate}
\item $M$ is \emph{causal} which means that it contains no causal loop;
\item for every distinct points $p$ and $q$ in $M$, the intersection $J^+(p) \cap J^+(q)$ is compact. 
\end{enumerate}
\end{definition}

There is another equivalent definition of globally hyperbolic spacetimes which uses the notion of \emph{Cauchy hypersurface}.

\begin{definition}
Let $M$ be a spacetime. A \emph{Cauchy hypersurface} is an acausal hypersurface $S$ in $M$ that intersect every inextensible causal curve in $M$ in exactly one point.
\end{definition}

In \cite{Choquet-Bruhat}, the following statement is proven.

\begin{theorem}
A spacetime is globally hyperbolic if and only if it contains a Cauchy hypersurface.
\end{theorem}

In a GH spacetime, the Cauchy hypersurfaces are homeomorphic one to the other. Consequently, if one of them is compact, all of them are compact.

\begin{definition}
A globally hyperbolic spacetime is \emph{spatially compact} (abbrev. CGH) if it contains a compact Cauchy hypersurface.
\end{definition}

The next definition introduces a notion of convexity from the point of view of causal structure of a spacetime.

\begin{definition}
A subset $U$ of a spacetime $M$ is \emph{causally convex} if every causal curve between two of its points is contained in $U$.
\end{definition}

\begin{proposition}\label{causally convex is GH}
A causally convex subset $U$ of a GH spacetime $M$ is GH.
\end{proposition}

\begin{proof}
The spacetime $M$ is causal, so $U$ contains no causal loop. Let $p \in U$. Since $U$ is causally convex, $J^+_U(p) = J^+(p) \cap U$. Let $q \in J^+_U(p)$. Since $M$ is GH, the intersection $J^+(p) \cap J^-(p)$ is compact. Thus, $J^+_U(p) \cap J^+_U(p) =  J^+(p) \cap J^-(p) \cap U$ is a compact subset of $U$. 
\end{proof}

\paragraph{\textit{Maximality.}} We introduce here the notion of maximality for conformal spacetimes. Let $M$ and $N$ be two globally hyperbolic conformal spacetimes.

\begin{definition}
A conformal embedding $f : M \longrightarrow N$ is a \emph{Cauchy embedding} if there exists a Cauchy hypersurface $S$ of $M$ such that $f(S)$ is a Cauchy hypersurface of $N$.
\end{definition}

\begin{definition}
A globally hyperbolic conformal spacetime $M$ is \emph{$\mathcal{C}$-maximal} if every Cauchy-embedding of $M$ into another globally hyperbolic conformal (resp. conformally flat) spacetime is onto.
\end{definition}

\subsection{The anti-de Sitter space} \label{1.2}

Let $\R^{2,n}$ the vector space $\R^{n+2}$, with coordinates $(u,v,x_1,\dots,x_n)$, endowed with the nondegenerate quadratic form of signature $(2,n)$
\begin{align*}
q_{2,n}(u,v,x_1,\dots,x_n) &:= -u^2-v^2+x_1^2+\dots+x_n^2.
\end{align*}
We denote by $<.,.>$ the associated symmetric bilinear form. For any subset $A$ of $\R^{2,n}$, we denote by $A^{\perp}$ the orthogonal of $A$ for $<.,.>$.

\begin{definition}
The anti-de Sitter space $AdS_{1,n}$ is the hypersurface $\{x \in \R^{2,n}:\ q_{2,n}(x) = -1\}$ endowed with the Lorentzian metric obtained by restriction of $q_{2,n}$.
\end{definition}

The tangent space on each point $x$ of $AdS_{1,n}$ coincides with $x^{\perp}$. In the coordinates $(r,\theta, x_1, \ldots, x_n)$ with 
$$
\left\lbrace \begin{array}{ccc}
              u & = & r \cos \theta \\
              v & = & r \sin \theta \\
             \end{array}
\right.             
$$
one can easily see that $AdS_{1,n}$ is diffeomorphic to $\mathbb{S}^1 \times \R^n$ and thus is oriented. Besides, a time-orientation is given by the vector field $\frac{\partial}{\partial \theta}$.

Anti-de Sitter space is a spacetime with constant sectional curvature equal to $-1$. Observe the analogy with the definition of hyperbolic space. As for hyperbolic space, geodesics of $AdS_{1,n}$ are intersections with $2$-planes (see \cite{salveminithesis} for the proof). The signature of the quadratic form $q_{2,n}$ restricted to a given $2$-plane $\mathcal{P}$ determines the nature of the associated geodesic. If it is equal to
\begin{itemize}
\item $(2,0)$ then $\mathcal{P}$  gives a \emph{timelike geodesic};
\item $(1,0,1)$ then $\mathcal{P}$ gives a \emph{lightlike geodesic};
\item $(1,1)$ then $\mathcal{P}$ gives a \emph{spacelike geodesic}.
\end{itemize}

As for the hyperbolic space, anti-de Sitter space has different models, namely, the Klein model and the conformal model.\\

\paragraph{\textit{The Klein model of the anti de-Sitter space.}} Let $\mathbb{S}(\R^{2,n})$ be the quotient of $\R^{2,n}$ by positive homotheties. Remark that $\mathbb{S}(\R^{2,n})$ is a double covering of $\mathbb{P}(\R^{2,n})$. We denote by $\pi$ the projection of $\R^{2,n}$ onto $\mathbb{S}(\R^{2,n})$.\\

The \emph{Klein model} $\mathbb{ADS}_{1,n}$ of the anti-de Sitter space is the projection of $AdS_{1,n}$ to $\mathbb{S}(\R^{2,n})$. The restriction of the projection $\pi$ to $AdS_{1,n}$ is one-to-one. We equip $\mathbb{ADS}_{1,n}$ with the Lorentzian metric such that $\pi$ is a one-to-one isometry between $AdS_{1,n}$ and $\mathbb{ADS}_{1,n}$.

The topological boundary $\partial \mathbb{ADS}_{1,n}$ of $\mathbb{ADS}_{1,n}$ is the projection of the lightlike cone onto $\mathbb{S}(\R^{2,n})$. The continous extension of the isometry between $AdS_{1,n}$ and $\mathbb{ADS}_{1,n}$ is a canonical homeomorphism betwen $AdS_{1,n} \cup \partial AdS_{1,n}$ and $\mathbb{ADS}_{1,n} \cup \partial \mathbb{ADS}_{1,n}$.

In this model, geodesics are the projection of $2$-planes. Moreover, the discussion above around the link between the type of the $2$-plane and the type of the geodesic it determines, still holds.

The projection of $AdS_{1,n}$ in the projective space $\mathbb{P}(\R^{2,n})$ also gives a model of the anti-de Sitter space. We denote it by $\mathtt{AdS}_{1,n}$. In this model, geodesics are still the projection of $2$-planes in the projective space. The Klein model $\mathbb{ADS}_{1,n}$ is a double covering of $\mathtt{AdS}_{1,n}$.\\ 

\paragraph{\textit{The conformal model of the anti de-Sitter space.}} We denote by
\begin{itemize}
\item $\mathcal{H}^{n}$ the upper hemisphere of the $n$-dimensional sphere $\mathbb{S}^n$;
\item ${d\theta}^2$ the standard Riemannian metric on $\mathbb{S}^1$;
\item $ds_0^2$ the restriction to $\mathcal{H}^n$ of the standard Riemannian metric on $\mathbb{S}^n$. 
\end{itemize}

\begin{proposition}\label{conformal model AdS}
The anti-de Sitter space $AdS_{1,n}$ is conformally isometric to $\mathcal{H}^{n} \times \mathbb{S}^1$ endowed with the Lorentzian metric $(ds_0^2 - d\theta^2)$.
\end{proposition}

A proof of this statement is given in \cite{barbot2015deformations}, Section $2$, Proposition $2.4$.

\subsection{The Einstein universe} \label{1.3}
We denote by $\pi$ the projection of $\R^{2,n}$ into $\mathbb{S}(\R^{2,n})$.
Let $C$ be the lightlike cone of $\R^{2,n}$. Let us consider the projection of $C$ in $\mathbb{S}(\R^{n+2})$. Since the metric $q_{2,n}$ is of signature $(2,n)$, the lightlike cone $C$ is connected and so do $\pi(C)$. Given two sections $\varphi, \varphi' : \pi(C) \longrightarrow C$, an easy computation shows that the two Lorentzian metrics on $\pi(C)$ obtained by the pull back of $q_{2,n}$ by $\varphi$ and $\varphi'$ are conformally equivalent. Therefore, $\pi(C)$ can be naturally endowed with a conformal class of Lorentzian metrics. The \emph{Einstein universe} is the space $\pi(C)$ endowed with this conformal class of Lorentzian metrics. We denote it by $Ein_{1,n-1}$. This is the \emph{Klein model} of the Einstein universe.

Now, let $\mathcal{S}$ be the Riemannian sphere of radius $\sqrt{2}$ of $\R^{2,n}$. The projection $\pi$ defines a one-to-one conformal isometry between $\mathcal{S} \cap C$, endowed with the restriction of the quadratic form $q_{2,n}$, and the Einstein universe. Besides, it is easy to see that $\mathcal{S} \cap C$ endowed with the restriction of the quadratic form $q_{2,n}$ is isometric to $\mathbb{S}^{n-1} \times \mathbb{S}^1$ endowed with the Lorentzian metric $(ds^2 - d\theta^2)$, where $ds^2$ and $d\theta^2$ are the standard Riemannian metrics on $\mathbb{S}^{n-1}$ and $\mathbb{S}^1$ respectively. Hence, the space $\mathbb{S}^{n-1} \times \mathbb{S}^1$ endowed with the conformal class $[ds^2 - d\theta^2]$ is conformally isometric to the Einstein universe. It is the \emph{conformal model} of the Einstein universe.

Remark that in the definition of the Einstein universe, we could consider the projection of the lightlike cone $C$ in the projective space $\mathbb{P}(\R^{2,n})$. It gives us another model of the Einstein universe  we denote by $\mathtt{Ein}_{1,n-1}$. Notice that $Ein_{1,n-1}$ is a double covering of $\mathtt{Ein}_{1,n-1}$.

According to Proposition \ref{conformal model AdS}, the anti-de Sitter space $AdS_{1,n}$ embeds conformally in the Einstein universe $Ein_{1,n}$. Moreover, its boundary $\partial AdS_{1,n}$ is the Einstein universe $Ein_{1,n-1}$. Actually, it follows from Proposition \ref{conformal model AdS} that $Ein_{1,n}$ is the union of two conformal copies of $AdS_{1,n}$ and of their conformal boundary $Ein_{1,n-1}$.\\

Each model of the Einstein universe gives simple descriptions of causal curves and lightlike geodesics (see \cite{Salvemini2013Maximal}, Section $2.2$). In the Klein model, lightlike geodesics are the projections of isotropic $2$-planes. Besides, the lightlike cone of a point $p$ is the intersection of $Ein_{1,n-1}$ with $p^{\perp}$. In the conformal model $\mathbb{S}^{n-1} \times \mathbb{S}^1$, lightlike geodesics are, up to parametrization, the curves $(x(t), e^{2i\pi t})$ where $x$ is a geodesic of the Riemannian sphere $\mathbb{S}^{n-1}$ parametrized by its arc length. Causal (timelike) curves are, up to parametrization, the curves $(x(t),e^{2i\pi t})$ where $x$ is a (stricly) $1$-Lipchitz map from $\R$ to the Riemannian sphere $\mathbb{S}^{n-1}$. 

Since the Einstein universe is compact, each point is causally related to any other point (see \cite{Salvemini2013Maximal}, Section $2.2$, Corollary $2$). Hence, the causal structure of Einstein space gives no information. However, the universal covering $\widetilde{Ein}_{1,n-1}$ has a rich causal structure, in particular, it is globally hyperbolic. The universal covering of $Ein_{1,n-1}$ is identified to $\mathbb{S}^{n-1} \times \R$ endowed with the conformal class $[ds^2 - dt^2]$ where $dt^2$ is the canonical metric over $\R$. Let $p : \widetilde{Ein}_{1,n-1} \longrightarrow Ein_{1,n-1}$ the universal covering map. Since it is a conformal map, the projection of any causal curve of $\widetilde{Ein}_{1,n-1}$ is a causal curve of $Ein_{1,n-1}$. Conversely, any causal curve of $Ein_{1,n-1}$ lifts to a causal curve of $\widetilde{Ein}_{1,n-1}$. It follows that causal (timelike) curves of $\widetilde{Ein}_{1,n-1}$ are, up to parametrization, the curves $(x(t), t)$ where $x$ is a (stricly) $1$-Lipchitz map from $\R$ to the Riemannian sphere $\mathbb{S}^{n-1}$. Moreover, lightlike geodesics are, up to parametrization, the curves $(x(t), t)$ where $x$ is a geodesic of the Riemannian sphere $\mathbb{S}^{n-1}$ parametrized by its arc length. Therefore, the future and the past of a point $(x_0,t_0)$ in $\widetilde{Ein}_{1,n-1}$ are given by
\begin{align*}
I^+ (x_0,t_0) &= \{(x,t) \in \mathbb{S}^{n-1} \times \mathbb{S}^1:\ d_0(x,x_0) < t - t_0\}\\
I^- (x_0,t_0) &= \{(x,t) \in \mathbb{S}^{n-1} \times \mathbb{S}^1:\ d_0(x,x_0) < t_0 - t\}
\end{align*}
where $d_0$ is the distance on the sphere. If we replace the strict inequalities by large inequalities in the sets above, we obtain the causal future and the causal past of $(x_0,t_0)$. Remark that for any point $(x_0,t_0)$ in $\widetilde{Ein}_{1,n-1}$, the subsets $J^+(x_0,t_0)$ and $J^-(x_0,t_0)$ are closed in $\widetilde{Ein}_{1,n-1}$. This doesn't hold if we replace the point $(x_0,t_0)$ by any subset of $\widetilde{Ein}_{1,n-1}$. However, it holds for compact subsets.

\begin{proposition}\label{future of a compact}
Let $K$ be a compact of $Ein_{1,n-1}$. Then, there are two continuous maps $f, g : \mathbb{S}^{n-1} \longrightarrow \R$ such that
\[J^+(K) = \{(x,t) \in \eu : \ f(x) \leq t\}\ \ \mathrm{and}\ \ J^-(K) = \{(x,t) \in \eu : t \leq g(x)\}.\]
\end{proposition}

\begin{proof}
Take \[f(x) = \inf_{(x_0,t_0) \in K} \{d_0(x,x_0) + t_0\}\ \ \mathrm{et}\ \ g(x) = \sup_{(x_0,t_0) \in K} \{t_0 - d_0(x,x_0)\}.\]
\end{proof}

\begin{remark} \label{pi}
Two points $(x,t)$, $(x',t')$ in $\eu$ are necessarily causally related if $|t - t'| \geq \pi$.
\end{remark}

Let consider the map $\bar{\sigma} : Ein_{1,n-1} \longrightarrow Ein_{1,n-1}$ defined as the product of the two antipodal maps of $\mathbb{S}^{n-1}$ and $\mathbb{S}^1$. It lifts to a map $\sigma : \eu \longrightarrow \eu$ which associates to $(x,t)$ the point $(-x,t+\pi)$. This is clearly a conformal map.

\begin{definition}
Two points of the universal Eintein universe are \emph{conjugate} if one is the image by $\sigma$ of the other.
\end{definition}

\begin{remark} \label{conjugate points}
The inextensible lightlike geodesics of $\eu$ starting from a point $p$ of $\eu$ have common intersections at all the points $\sigma^k(p)$ for $k$ in $\mathbb{Z}$. Outside these points, they are pairwise disjoint. 
\end{remark}

\begin{remark}\label{fundamental group}
The fundamental group of $Ein_{1,n-1}$ is isomorphic to $\mathbb{Z}$ and is identified with the cyclic group generated by the map $\delta = \sigma^2$.\\
The fundamental group of $\mathtt{Ein}_{1,n-1}$ is also isomorphic to $\mathbb{Z}$ and is identified with the cyclic group generated by the map $\sigma$.\\
\end{remark}

It easily follows from the description of causal curves that $\eu$ is globally hyperbolic with Cauchy hypersurfaces isomorphic to $\mathbb{S}^{n-1}$.

\begin{definition}\label{affine domain}
For every $p_0 \in AdS_{1,n}$, we define the \emph{affine domain} $U(p_0)$ associated to $p_0$ as the intersection of $Ein_{1,n-1}$ with the projection in $\mathbb{S}(\R^{2,n})$ of the half space \mbox{$\{p \in \R^{2,n}:\ <p,p_0>\  < 0\}$}.
\end{definition}

\begin{remark}
If we take $p_0 = (1,0,\ldots,0)$, we observe that the affine domain $U(p_0)$ identifies with the open subset $\{(x,e^{2i\pi t}) \in \mathbb{S}^{n-1} \times \mathbb{S}^1:\ 0 < t < \pi\}$ of the conformal model of the Einstein universe.
\end{remark}

\begin{definition}\label{affine domain in the universal covering}
For every $t \in \R$, the subset $\mathbb{S}^{n-1} \times ]t,t+\pi[$ of $\eu$ is an \emph{affine domain} of $\eu$. 
\end{definition}

\begin{remark}
The universal covering $p$ of the Einstein universe is one-to-one on every affine domain and the image is an affine domain of $Ein_{1,n-1}$.
\end{remark}

\begin{proposition}\label{acausal set affine domain}
Any acausal subset $\Lambda$ of $\eu$ is contained in an affine domain of $\eu$.
\end{proposition}

\begin{proof}
Let us consider the projection $p_2 : \mathbb{S}^{n-1} \times \R \longrightarrow \R$. According to Remark \ref{pi}, for every distinct points $(x,t)$, $(x',t')$ in $\Lambda$, we have $|t - t'| < \pi$. Therefore, $p_2(\Lambda)$ is bounded; besides, $m = \inf p_2(\Lambda)$ and $M = \sup p_2(\Lambda)$ satisfy $M - m \leq \pi$. Thus, there exists $t_0 \in [m,m+\pi]$ such that $\Lambda \subset \mathbb{S}^{n-1} \times ]t_0, t_0+\pi[$. 
\end{proof}

In $\eu \simeq (\mathbb{S}^{n-1} \times \R ,[ds^2 - dt^2])$, every achronal (resp. acausal) subset $A$ is precisely the graph of a $1$-Lipschitz (resp. $1$-contracting) function $f: \Lambda_0 \to \R$ where $\Lambda_0$ is a subset of $\mathbb{S}^{n-1}$ endowed with its canonical metric $d_0$. The subset $A$ is closed if and only if $\Lambda_0$ is closed. In particular, the achronal (resp. acausal) embedded topological hypersurfaces are exactly the graphs of $1$-Lipschtiz (resp. $1$-contracting) functions $f : \mathbb{S}^{n-1} \to \R$: they are topological $(n-1)$-spheres.\\

Although there is no acausal and achronal subsets of $Ein_{1,n-1}$, we can ask the following question: 
\begin{center}
Can two distinct points of $Ein_{1,n}$ be lifted to points of $\eu$ which are not extremities of a causal curve?
\end{center}

The following proposition gives an answer to this question.

\begin{proposition}\label{P}
Two distinct points $p$ and $q$ of $Ein_{1,n-1}$ can be lifted to points $\tilde{p}$ and $\tilde{q}$ of $\eu$ respectively which are not extremities of a causal (resp. timelike) curve if and only if the sign of $<p,q>$ is \emph{negative} (resp. \emph{non-positive}).
\end{proposition}

This is a consequence of Lemma $10.13$ in \cite{andersson2012}. From Proposition \ref{P}, one can define acausal (resp. achronal) subsets of $Ein_{1,n-1}$ as the subsets which can be lifted to acausal (achronal) subsets of $\eu$.

\begin{definition}
A subset $\Lambda$ of $Ein_{1,n-1}$ is achronal (resp. acausal) if for every distinct points $x,y$ in $\Lambda$, the scalar product $<x,y>$ is non-positive (resp. negative). 
\end{definition}

\begin{remark}
A subset $\Lambda$ of $Ein_{1,n-1} = \partial AdS_{1,n}$ is acausal if and only if every distinct points $x,y$ in $Ein_{1,n-1}$ can be related with a spacelike geodesic of anti-de Sitter space.
\end{remark}

\subsection{Isometry groups}\label{1.4}

Let us denote by $O(2,n)$ the group of linear isometries of $\R^{2,n}$ for the quadratic form $q_{2,n} = -u^2 + v^2 + x_1^2 + \ldots + x_n^2$ and by $O_0(2,n)$ the identity component of $O(2,n)$.\\

\paragraph{\textit{Isometries of anti-de Sitter space.}} It is clear that $O(2,n)$ acts on the anti-de Sitter space $AdS_{1,n}$ by isometries. It turns out that every isometry of $AdS_{1,n}$ comes from an element of $O(2,n)$.

\begin{proposition}\label{isometry group of AdS}
The isometry group of $AdS_{1,n}$ is $O(2,n)$. Moreover, the group of orientation and time-orientation preserving isometries of $AdS_{1,n}$ is $O_0(2,n)$.
\end{proposition}

The proof of this proposition uses the following lemma.

\begin{lemma}[Proposition $2.1$, Section $2$, Chap. $1$ in \cite{salveminithesis}]\label{lemma isometry}
Let $M$ be a Lorentzian manifold. If $f : M \to M$ is an isometry which fixes a point $x \in M$ such that $d_xf$ is the identity on $T_xM$, then $f$ is the identity on $M$.
\end{lemma}

\begin{proof}[Proof of Proposition \ref{isometry group of AdS}]
It is clear that $O(2,n)$ acts transitively on the frame bundle $\mathcal{F}(M)$ of $M$. Since $O(2,n) \subset Isom(AdS_{1,n})$, the action of $Isom(AdS_{1,n})$ on $\mathcal{F}(M)$ is also transitive. Moreover, by Lemma \ref{lemma isometry}, this action is free; therefore so is the action of $O(2,n)$. It follows that $Isom(AdS_{1,n}) = O(2,n)$. The identity component $O_0(2,n)$ is made of isometries which preserves the orientation of $\R^{2,n}$ and the orientation of its timelike plans. It follows that $O_0(2,n)$ is the group of orientation and time-orientation preserving isometries of $AdS_{1,n}$.
\end{proof}

\paragraph{\textit{Conformal isometries of Einstein space.}} In Riemannian geometry , \mbox{Liouville's} theorem states that a conformal mapping between two open subsets of the sphere $\mathbb{S}^n$ (for $n \geq 3$) is the restriction of a Möbius transformation. A consequence of this theorem is that the conformal isometry group of the sphere $\mathbb{S}^n$ is isomorphic to $O(1,n)$. Similarly, in Lorentzian geometry, there is an analogue of Liouville's theorem.

\begin{theorem}[Liouville]\label{Liouville}
A conformal mapping between two open subsets of the Einstein universe $Ein_{1,n-1}$ (for $n \geq 3$) is the restriction of an element of $O(2,n)$.
\end{theorem}

A proof of a general version of this theorem stated in the setting of pseudo-Riemannian metrics is given by C. Frances in \cite{francesarticle} Chap. $2$. The orthogonal group $O(2,n)$ acts by conformal isometries on $Ein_{1,n-1}$ (see Lemma $2.6$, Section $2.3$, Chap. $2$ of \cite{salveminithesis}). An immediate consequence of this fact and Liouville's theorem is the following proposition. 

\begin{proposition}
The conformal isometry group of $Ein_{1,n-1}$ (for $n \geq 3$) is isomorphic to $O(2,n)$. Moreover, the group of orientation and time-orientation preserving conformal isometries of $Ein_{1,n-1}$ is isomorphic to $O_0(2,n)$.
\end{proposition}

\paragraph{\textit{Conformal isometries of universal Einstein space.}} Let us denote by $\tilde{O}(2,n)$ the group of conformal transformations of $\eu$ and let $\tilde{O}_0(2,n)$ be the group of orientation and time-orientation preserving conformal transformations of $\eu$.\\

Every conformal transformation $f$ of $Ein_{1,n-1}$ lifts to a conformal transformation $\tilde{f}$ of $\eu$ such that $p \circ \tilde{f} = f \circ p$. Conversely, every conformal transformation of $\eu$ defines a conformal transformation of the quotient space $Ein_{1,n-1} = \eu / <\delta>$. Thus, there is a surjective morphism $j~ : \tilde{O}(2,n) \longrightarrow O(2,n)$ and its kernel is generated by $<\delta>$. Notice that the image of $\tilde{O}_0(2,n)$ by this morphism is $O_0(2,n)$. It must be clear to the reader that $\tilde{O}(2,n)$ (resp. $\tilde{O}_0(2,n)$) is not the universal covering of $\tilde{O}(2,n)$ (resp. $\tilde{O}_0(2,n)$).\\

Throughout this paper, we will only consider isometries and conformal isometries which preserves orientation and time-orientation, i.e. elements of $O_0(2,n)$ or $\tilde{O}_0(2,n)$.

\subsection{Conformally flat spacetimes} \label{1.5}

A spacetime $M$ is conformally flat if it is locally conformally isometric to the Minkowski space. It turns out that the Einstein universe is conformally flat. This comes from the fact that the flat Minkowski spacetime conformally embeds in the Einstein universe (see \cite{salveminithesis}, Chap. $2$, Section $2.7$). Since the action of the group $O(2,n)$ on $Ein_{1,n-1}$ is transitive, every point of the spacetime $Ein_{1,n-1}$ has a neighborhood conformally equivalent to $\R^{1,n-1}$. It follows that any spacetime locally modeled on $Ein_{1,n-1}$, i.e. equipped with a $(O_0(2,n), Ein_{1,n-1})$-structure, is conformally flat. By the Lorentzian version of Liouville theorem (see Theorem \ref{Liouville}), the inverse is true in dimension at least $3$. More precisely, any conformally flat spacetime of dimension equal or greater than $3$ is locally modeled on the Einstein universe. Remark that since the anti-de Sitter space conformally embeds in the Einstein spacetime, it is conformally flat.\\

\paragraph{\textit{Maximality.}} We defined in Section \ref{1.1} a notion of maximality for globally hyperbolic \emph{conformal} spacetimes. In fact, we can extend this definition to any category of spacetimes (see \cite{Salvemini2013Maximal}, Section $1$, Definition $2$), in particular for the category $\mathcal{C}_0$ of conformally flat spacetimes. 

\begin{definition}
A globally hyperbolic conformally flat spacetime $M$ is \emph{$\mathcal{C}_0$-maximal} if every Cauchy-embedding of $M$ into another globally hyperbolic conformally flat spacetime is onto.
\end{definition}

Since conformally flat spacetimes are conformal spacetimes, every $\mathcal{C}$-maximal spacetime is $\mathcal{C}_0$-maximal. Conversely, C. Rossi proves in \cite{salveminithesis} that every $\mathcal{C}_0$-maximal spacetime is $\mathcal{C}$-maximal.

\section{Anosov representations} \label{2}

In this section, we recall the notions of limit set and Anosov representations. Let $\mathcal{B} = \{e_1,\ldots, e_{n+2}\}$ be a basis of $\R^{2,n}$ such that for every $x \in R^{2,n}$ with coordinates $(x_1, \ldots, x_{n+2})$ in the basis $\mathcal{B}$, we have 
\[q_{2,n}(x) = x_1x_{n+2} + x_2x_{n+1} + x_3^2 + \ldots + x_{n+2}^2.\]

\subsection{Dynamic on the projective space} \label{2.1}

\paragraph{\textit{Cartan decomposition.}} Let $K = O(n+2) \cap O_0(2,n)$ be a maximal compact subgroup of $O_0(2,n)$ which is isomorphic to $SO(2) \times SO(n)$. Let $A^+$ be the Weyl chamber of $O_0(2,n)$ defined as follow

\[A^+ := \{ \left(
\begin{array}{ccccccc}
e^{\lambda} &  &  &  &  &  & \\
 & e^{\mu} &  &  &  & & \\
 &  & 1 & & & &  \\
 &  &  & \ddots & & &  \\
 &  &  &  & 1 &  & \\
 &  &  &  &  & e^{-\mu} &  \\
 &  &  &  &  &  & e^{-\lambda}
\end{array} \right)
;\ \lambda \geq \mu \geq 0\} \]

The group $O_0(2,n)$ admits the \emph{Cartan decomposition} $O_0(2,n) = K A^+ K$: any $g$ in $O_0(2,n)$ may be written
\[ g = k_g a_g l_g\]
for some $k,l$ in $K$ and a unique $a$ in $A^+$ called \emph{Cartan projection} of $g$ (see \cite{helgason1979differential}, Chap. IX, Thm $1.1$).\\

\paragraph{\textit{$P_1$-divergence.}} Let $\{g_i\}_i$ be a sequence of $O_0(2,n)$. For every $i \in \mathbb{N}$, we denote by $\lambda_i$ and $\mu_i$ the exponents appearing on the diagonal of the Cartan projection of $g_i$.

\begin{definition}
The sequence $\{g_i\}_i$ is \emph{$P_1$-divergent} if $\lim (\lambda_i - \mu_i) = + \infty$.  
\end{definition}

Let us describe the dynamic on the projective space $\mathbb{P}(\R^{2,n})$ under the action of $P_1$-divergent sequences of $O_0(2,n)$ . We denote by $[x]$ the equivalence class of a vector $x \in \R^{2,n}$ in $\mathbb{P}(\R^{2,n})$. For every $i \in \mathbb{N}$, let $a_i$ be the Cartan projection of $g_i$.

\begin{proposition}
The sequence $\{a_i\}_i$ converge uniformly on compact subsets of the complementary of $[e_{n+2}]^{\perp}$ towards the constant map $[e_1]$.
\end{proposition}

\begin{proof}
We denote by $||.||$ the euclidean norm on $\R^{n+2}$. We consider on the projective space the distance defined for every lines $p$, $p'$ with respective direction vectors $\mathrm{x}$ and $\mathrm{x}'$ by
\[d(p,p') = \min\{||\mathrm{x}-\mathrm{x}'||, ||\mathrm{x}+\mathrm{x}'||\}.\]
We have $[e_{n+2}]^{\perp} = \{[x_1:\ldots:x_{n+2}],\ x_1 = 0\}$. Let $C$ be a compact of $\pr(\R^{2,n})$ disjoint from $[e_{n+2}]^{\perp}$. There exist $\eta > 0$ such that for every $[1:x_2:\ldots:x_{n+2}] \in C$, we have $|x_i| \leq \frac{1}{\eta}$ for $i = 2,\ldots, n+2$. It follows that for every $p = [1:x_2:\ldots:x_{n+2}] \in C$
\[d(a_i.p, [e_1]) = d(e^{-\lambda_i}a_i.p, [e_1]) \leq \frac{1}{\eta^2}(e^{2(\mu_i-\lambda_i)}+ne^{-2\lambda_i}+e^{-2(\mu_i+\lambda_i)}+e^{-4\lambda_i}).\]
Thus, $\lim \underset{p \in C} {\sup} d(a_i.p, [e_1]) = 0$.
\end{proof}

\begin{corollary}\label{dyn. pr}
There is a subsequence $\{g_{i_j}\}_j$ such that
\begin{itemize}
\item for every $j \in \mathbb{N}$, $g_{i_j} = k_{i_j} a_{i_j} l_{i_j}$ where $\{k_{i_j}\}$ and $\{l_{i_j}\}$ are sequences of $K$ converging on the projective space towards maps $k, l$ in $K$;
\item the sequence $\{g_{i_j}\}_j$ converge uniformly on compact subset of the complementary of the $q_{2,n}$-orthogonal of $p_- = l^{-1} [e_{n+2}]$ towards the constant map $p_+ = k [e_1]$.
\end{itemize}
\end{corollary}

The points $p_+$ and $p_-$ are respectively called the \emph{attracting} and the \emph{repelling} points of $\{g_i\}_i$.\\

Since $e_1$ and $e_{n+2}$ are isotropic vectors of $\R^{2,n}$, the points $[e_1]$ and $[e_{n+2}]$ belong to $\mathtt{Ein}_{1,n-1}$. Moreover, since the elements $k, l$ lies in $K = O_0(2,n) \cap O(n+2)$, the points $p_+ = k [e_1]$ and $p_- = l^{-1} [e_{n+2}]$ still belong to $\mathtt{Ein}_{1,n-1}$. Recall that in the Einstein space, $p_-^{\perp}$ is the lightlike cone of $p_-$. Hence, Corollary \ref{dyn. pr} can be rephrased in the Einstein universe as follow.

\begin{corollary}\label{dyn. pr ein}
There is a subsequence $\{g_{i_j}\}_j$ and there are points $p_+$ and $p_-$ in $\mathtt{Ein}_{1,n-1}$ such that $\{g_{i_j}\}_j$ converge uniformly on compact subsets outside the lightlike cone of $p_-$ towards the constant map $p_+$.
\end{corollary}

\paragraph{\textit{Limit set.}} Let $H$ be a discrete subgroup of $O_0(2,n)$.

\begin{definition}
The \emph{limit set} $\Lambda_H$ of $H$ is the set of the attracting and repelling points of $P_1$-divergent sequences of $H$.
\end{definition}

The limit set is well-studied in \cite{guivarc'h1990}. In particular, the following proposition is proven.

\begin{proposition}
The limit set $\Lambda_H$ of $H$ is a closed, $H$-invariant subset of $\mathtt{Ein}_{1,n-1}$.
\end{proposition}

\subsection{Anosov-representations} \label{2.2}

Let $\Gamma$ be a hyperbolic group and $\partial \Gamma$ its Gromov boundary. We do not recall here the definition of hyperbolic group; we direct the reader to \cite{Gueritaud2017Anosov}, Section $2.1$ for the basics of Gromov hyperbolic groups' theory. Recall that $\Gamma$ is a convergence group for its action on $\partial \Gamma$: for any divergent sequence $\{\gamma_n\}_n$ of $\Gamma$, there exist a subsequence $\{\gamma_{n_k}\}_k$ and points $\xi_+$ and $\xi_-$ in $\partial \Gamma$ such that $\{\gamma_{n_k}\}_k$ converge uniformly on compact subsets of $\partial \Gamma \backslash \{\xi_-\}$ towards the constant map $\xi_+$ (see \cite{Kapovich2002Boundaries}). The following definition of Anosov representation is not the original one from \cite{labourie} but an equivalent one taken from \cite{guichard2015tameness} and rephrased in the setting of the Lorentzian geometry.

\begin{definition}
A representation $\rho : \G \longrightarrow O_0(2,n)$ is $P_1$-Anosov if
\begin{enumerate}
\item every sequence of pairwise distinct elements in $\rho(\G)$ is $P_1$-divergent;
\item there exist a continuous, $\rho$-equivariant boundary map $\xi : \partial \G \longrightarrow Ein_{1,n-1}$ that is
\begin{enumerate}
\item \emph{transverse}, meaning that every pair of distinct points $(\eta, \eta')$ in $\partial \G$ is sent to the pair $(\xi(\eta),\xi(\eta'))$ in $Ein_{1,n-1} \times Ein_{1,n-1}$ such that $\xi(\eta)$ is not contained in the lightlike cone of $\xi(\eta')$;
\item \emph{dynamics-preserving}, meaning that if $\eta$ is the attracting fixed point of some element $\gamma \in \G$ in $\partial \G$, then $\xi(\eta)$ is an attracting fixed point of $\rho(\gamma)$ in $Ein_{1,n-1}$. 
\end{enumerate}
\end{enumerate} 
\end{definition}

\begin{proposition}[\cite{Gueritaud2017Anosov}]
If a representation $\rho : \G \longrightarrow O_0(2,n)$ is $P_1$-Anosov with boundary map \mbox{$\xi : \partial \G \longrightarrow \mathtt{Ein}_{1,n-1}$}, then $\Lambda_{\rho(\G)} = \xi(\partial \G)$.
\end{proposition}

\section{Causal geodesic space and Anosov representations} \label{3}

In this section, we introduce the \emph{causal geodesic space} and study the dynamic on this space under actions of Anosov representations. In particular, we give a geometrical proof of Theorem \ref{GKW0}.

\subsection{Causal geodesic space} \label{4.1}

We define here the \emph{causal geodesic space}.\\

Let $M$ be a spacetime (oriented and time-oriented).

\begin{definition} \label{limit curve}
A curve $\varphi$ in $M$ is a \emph{limit curve} of a sequence $\{\varphi_i\}_i$ of curves in $M$ if for every point $p$ in $\varphi$, every neighborhood of $p$ meets all the curves $\varphi_i$ except a finite number.
\end{definition}

Let us consider the set $\mathcal{G}(M)$ consisting of geodesics of $M$. We equip it with the topology for which every convergent sequence is convergent in the sense of Definition \ref{limit curve}. In this Section, we denote by $\mathcal{P}(M)$ the subspace of lightlike geodesics of $M$. 

\begin{proposition} \label{photons Ein}
The space $\mathcal{P}(\mathtt{Ein_{1,n}})$ is a smooth manifold diffeomorphic to the unit tangent bundle $T^1\mathbb{S}^n$.
\end{proposition}

\begin{proof} Recall that $Ein_{1,n}$ is the double covering of $\mathtt{Ein_{1,n}}$ (see Section \ref{1.3}). Remark that $\mathcal{P}(\mathtt{Ein_{1,n}})$ is homeomorphic to $\mathcal{P}(Ein_{1,n})$. We prove that $\mathcal{P}(Ein_{1,n})$ is homeomorphic to $T^1\mathbb{S}^n$. Let $\varphi \in \mathcal{P}(Ein_{1,n})$. Up to parametrization, $\varphi$ is the curve $(x(t), e^{it})$. Let $f$ be the map which associates to every $\varphi(t) = (x(t), e^{it})$ the point $(x(0), x'(0)) \in T^1\mathbb{S}^n$. For every $(x, e^{it}) \in Ein_{1,n}$, the tangent space $T_{(x,e^{it})}Ein_{1,n}$ is isomorphic to $T_x\mathbb{S}^n \oplus T_{e^{it}}\mathbb{S}^1$. Let $u(t)$ be a unit vector field on $\mathbb{S}^1$. We define the map $g$ which associates to every $(x,v) \in T^1\mathbb{S}^n$ the unique lightlike geodesic going through $(x,1)$ and tangent to $(u(0)+v)$. Clearly, the maps $f$ and $g$ are continuous and $g = f^{-1}$. We equip $\mathcal{P}(Ein_{1,n})$ with the differential structure for which $f$ is a diffeomorphism.
\end{proof}

\begin{remark}
It follows from Proposition \ref{photons Ein} that $\mathcal{P}(\mathtt{Ein_{1,n}})$ is compact.
\end{remark}

Let $\overline{\mathtt{AdS}}_{1,n} := \mathtt{AdS}_{1,n} \cup \partial \mathtt{AdS}_{1,n}$.

\begin{definition}
The \emph{causal geodesic space} of $\overline{\mathtt{AdS}}_{1,n}$, denoted by $\mathcal{C}$, is the subspace of $\mathcal{G}(\bar{\mathtt{AdS}_{1,n}})$ consisting of timelike and lightlike geodesics of $\mathtt{AdS}_{1,n}$ and lightlike geodesics of $\mathtt{Ein}_{1,n-1} = \partial \mathtt{AdS}_{1,n}$. 
\end{definition}

Let us consider the closed subspace
\[\bar{X} = \{P \in Gr_{2}(\R^{n+2})\ :\ <x,x>\ \leq 0,\ \forall x \in P\}\]
in the Grassmannian space $Gr_{2}(\R^{n+2})$.

\begin{remark}
In this section, we consider on $Gr_{2}(\R^{n+2})$ the distance $\delta$ defined as follow. Let $d$ be the distance on the projective space $\pr(\R^{2,n})$ defined  for every lines $\ell, \ell'$ with respective direction vectors $\mathrm{x}, \mathrm{x'}$ as the angle between $\mathrm{x}$ and $\mathrm{x}'$ in $\R^{n+2}$.\\
We denote by $Hd$ the Hausdorff distance. For every $P, P' \in Gr_{2}(\R^{n+2})$, we set
\[\delta (P,P') := Hd(\pr(P), \pr(P')).\]
where $\pr : \R^{2,n} \backslash \{0\} \longrightarrow \pr(\R^{2,n})$ is the canonical projection. 
\end{remark}

\begin{proposition}
The space $\bar{X}$ and the causal geodesic space are homeomorphic. In particular, the topological boundary of $\bar{X}$ in $Gr_{2}(\R^{n+2})$ is a manifold homeomorphic to the space $\mathcal{P}(\mathtt{Ein}_{1,n})$ of lightlike geodesics of $\mathtt{Ein}_{1,n}$.
\end{proposition}

\begin{proof}
Let $f : \bar{X} \longrightarrow \mathcal{C}$ be the mapping which associates to every $2$-plane $P$ in $\bar{X}$ the geodesic $\varphi$ defined as the intersection of $\pr(P)$ with $\overline{\mathtt{AdS}}_{1,n}$. By definition of geodesics of $\mathtt{AdS}_{1,n}$ and $\mathtt{Ein}_{1,n}$, the map $f$ is one-to-one. Let $\{P_i\}_i$ be a sequence of $\bar{X}$ converging to some $P \in \bar{X}$. We denote by $\{\varphi_i\}_i$ and $\varphi$ the images of $\{P_i\}_i$ and $P$ respectively. Let $p \in \varphi$ and $U$ an open neighborhood of $p$. There exists $\mathrm{x} \in P$ such that $p = \pr(\mathrm{x})$. Since $\lim \delta(P_i, P) = 0$, we have $\lim d(\pr(P_i), \pr(\mathrm{x})) = 0$. But $\pr(P_i)$ is compact in $Gr_{2}(\R^{n+2})$, thus there exists $\mathrm{x}_i \in P_i$ such that $d(\pr(P_i), \pr(\mathrm{x})) = d(\pr(\mathrm{x}_i),\pr(\mathrm{x}))$. Therefore, $\pr(\mathrm{x}_i) = p_i \in \varphi_i$ converges to $p$. It follows that $U$ meets all the curves $\varphi_i$ except a finite number. We deduce that the map $f$ is continous. Since $\bar{X}$ is compact, it follows that $f$ is a homeomorphism.\\

The topological boundary of $\bar{X}$ in $Gr_{2}(\R^{n+2})$ consists of:
\begin{enumerate}
\item isotropic $2$-planes, corresponding to lightlike geodesics of $\mathtt{AdS}_{1,n}$;
\item totally isotropic $2$-planes, corresponding to lightlike geodesics of $\mathtt{Ein}_{1,n-1}$.
\end{enumerate}
Thus, $\partial \bar{X}$ is homeomorphic to $\mathcal{P}(\overline{\mathtt{AdS}}_{1,n})$. Let us consider the continuous map $p : \mathtt{Ein}_{1,n} \longrightarrow \overline{\mathtt{AdS}}_{1,n}$ defined by
$p([u:v:x_1:\ldots:x_n:x_{n+1}]) = [u:v:x_1:\ldots:x_n]$. Notice that a point $[u:v:x_1:\ldots:x_n]$ in $\mathtt{AdS}_{1,n}$ (assume $-u^2-v^2+x_1^2+\ldots+x_n^2 = -1$) has two preimages $[u:v:x_1:\ldots:x_n:1]$ and $[u:v:x_1:\ldots:x_n:-1]$. Moreover, a point $[u:v:x_1:\ldots:x_n]$ in $\partial \mathtt{AdS}_{1,n}$ has one preimage $[u:v:x_1:\ldots:x_n:0]$. In fact, $p$ is a ramified covering of degree $2$. It induces a continous map $\bar{p}$ from $\mathcal{P}(\mathtt{Ein}_{1,n})$ to $\mathcal{P}(\overline{\mathtt{AdS}}_{1,n})$. The map $\bar{p}$ is clearly one-to-one and since $\mathcal{P}(\mathtt{Ein}_{1,n})$ is compact, $\bar{p}$ is a homeomorphism.
\end{proof}

\subsection{Dynamic on the causal geodesic space}

In this section, we study actions of $P_1$-Anosov representations on the causal geodesic space $\mathcal{C}$. In the previous section, we proved that $\mathcal{C}$ is homeomorphic to the closed subset $\bar{X}$ of $Gr_2(\R^{n+2})$. In \cite{guichard2015tameness}, it is proved that images of $P_1$-Anosov representations act properly discontinously and cocompactly on $\bar{X}$ from which we removed a "bad set" (see \cite{guichard2015tameness}, Theorem $4.1$). We rephrase here this result in terms of causal geodesics and give a geometrical proof of it.

\begin{theorem} \label{GKW}
Let $\Gamma$ be a hyperbolic group and $\rho : \Gamma \longrightarrow O_0(2,n)$ a $P_1$-Anosov representation with boundary map $\xi : \partial \Gamma \longrightarrow Ein_{1,n-1}$. Let
\begin{align*}
U &= \{\varphi \in \mathcal{C}\ :\ \varphi \cap \xi(\partial \Gamma) = \emptyset\}
\end{align*}
be the subspace consisting of causal geodesics avoiding the limit set $\xi(\partial \Gamma)$.\\
Then, the action of $\rho(\Gamma)$ on $U$ is properly discontinuous and cocompact.
\end{theorem}

Properness is proved in Section \ref{proper disc} and cocompactness in Section \ref{compactness}. In what follows, we denote by $\Lambda$ the limit set of $\rho(\G)$.

\subsubsection{Proper discontinuity} \label{proper disc}

Two points $\varphi$, $\varphi'$ in $\mathcal{C}$ are said to be \emph{dynamically related} if there exist a sequence $\{\varphi_i\}_i$ in $\mathcal{C}$ converging to $\varphi$ and a sequence $\{g_i\}_i$ in $\rho(\Gamma)$ going to infinity (i.e. leaving every finite subset of $\rho(\G)$) such that the sequence $\{g_i.\varphi_i\}_i$ converges to $\varphi'$. Recall the following general dynamical criterion of properness (see \cite{francesproperness}, Section $3.2$, Proposition $1$ for a proof). 

\begin{criterion}
A group $G$ acts properly discontinously on a Hausdorff topological space $Y$ if and only if no pairs of points of $Y$ are dynamically related.
\end{criterion}

Using this criterion, the proper discontinuity in theorem \ref{GKW} is a direct consequence of the following lemma. Recall that $U$ denote the set of causal geodesics which avoid the limit set $\Lambda$.

\begin{lemma}
Let $\{\varphi_i\}_i$ be a sequence of $U$ converging to some $\varphi$ in $U$ and let $\{g_i\}_i$ be a $P_1$-divergent sequence of $\rho(\G)$ such that $\{g_i.\varphi_i\}_i$ converges to some $\varphi'$ in $\mathcal{C}$. Then, $\varphi'$ meets the limit set.  
\end{lemma}

\begin{proof} 
Since $\{g_i\}_i$ is $P_1$-divergent, there exist an attracting point $p_+$ and a repelling point $p_-$ in $\Lambda \subset \mathtt{Ein}_{1,n-1}$. The curve $\varphi$ avoid the limit set, in particular $p_- \not \in \varphi$. It follows that there exists $p \in \varphi$ such that $p$ is not in the lightlike cone of $p_-$. Since $\{\varphi_i\}_i$ converges toward $\varphi$, there exists a sequence $\{p_i\}_i$ converging to $p$ such that every $p_i$ is a point of $\varphi_i$. Let $K$ be a compact neighborhood around $p$ in $\mathtt{AdS}_{1,n} \cup \partial \mathtt{AdS}_{1,n}$, disjoint from the lightlike cone of $p_-$. The compact $K$ contains all the points $p_i$ except a finite number. Therefore, the sequence $\{g_i.p_i\}_i$ converges to $p_+$. By the convergence of $\{g_i.\varphi_i\}_i$ to $\varphi'$, the attracting point $p_+$ belongs to $\varphi'$ and hence $\varphi' \not \in \mathcal{C}$.
\end{proof}

\subsubsection{Compactness} \label{compactness}

We use the following dynamical compactness criterion from \cite{kapovich2013dynamics}, inspired by Sullivan's dynamical characterization of convex cocompactness.

\begin{criterion}[\cite{kapovich2013dynamics}, Proposition $2.5$] \label{compactness criterion}
Let $G$ be a group acting by homeomorphisms on a compact metric space $(Z,d_Z)$ and on a compact set $K$. Let $F$ be a closed $G$-invariant subset of $Z$ fibering equivariantly over $K$, with fibers denoted by $F_p$, $p \in K$.\\
Suppose that for any $p \in K$, there exist $g \in G$ which locally expand distances from a point to a fiber. More precisely, there exist an open set $W_p \subset Z$ containing $F_p$ and a constant $c_p > 1$ such that
\begin{align} \label{condition criterion}
d_Z(g.z,g.F_q) \geq c_p.d_Z(z, F_q)
\end{align}
for every $z \in W_p$ and $q \in K$ with $F_q \subset W_p$.\\
Then, the action of $G$ on $Z \backslash F$ is cocompact.
\end{criterion}

By this criterion, the compactness in theorem \ref{GKW} is a consequence of the following proposition. Let us first introduce a notation. For every $p \in \pr(\R^{n+2})$, let
\[F_p := \{ P \in Gr_2(\R^{n+2});\ p \in \pr(P)\}\]
be the set consisting of projective lines going through $p$.

\begin{proposition} \label{techn. prop}
Let $\rho : \G \longrightarrow O_0(2,n)$ a $P_1$-Anosov representation with boundary map $\xi: \partial \G \longrightarrow \mathtt{Ein}_{1,n-1}$. We denote by $\Lambda$ the limit set $\xi(\partial \G)$.\\
Then, for every $p \in \Lambda$ and every $c > 1$, there exist $g \in \rho(\G)$ and an open subset $W_p$ of $Gr_2(\R^{n+2})$ containing $F_p$ such that
\[\delta(g.P,g.F_q) \geq c.\delta(P, F_q)\]
for every $P \in W_p$ and every $q \in \pr(\R^{n+2})$ with $F_q \subset W_p$.
\end{proposition}

The proof of this proposition uses the two following facts. Let us first introduce some notations.
Let $\{e_1, \ldots, e_{n+2}\}$ be a basis of $\R^{2,n}$ such that $e_1$ and $e_{n+2}$ are isotropic vectors of $\R^{2,n}$. We set $p_0 = \pr(e_1)$. We denote by $B_r$ and $\bar{B}_r$ respectively the open and closed balls of radius $r$ centered in $p_0$ of the projective space.\\

The following identity is easily established.

\begin{tech lemma} \label{techn. lemma 1}
For every $q \in \pr(\R^{n+2})$ and every $P \in Gr_2(\R^{n+2})$, we have
\[\delta(P, F_q) = d(q, \pr(P)).\]
\end{tech lemma}

\begin{tech lemma} \label{techn. lemma 2}
For every ball $B_{\varepsilon}$, there exist a ball $\bar{B}_r$ containing stricly $B_{\varepsilon}$ and a real number $M > 0$ such that 
\begin{align*}
d(q, \pr(P) \cap \bar{B}_r) \leq M.d(q,\pr(P))
\end{align*}
for every $q \in B_{\varepsilon}$ and every $P \in Gr_2(\R^{n+2})$ such that the projective line $\pr(P)$ meets $\bar{B}_r$.
\end{tech lemma}

\begin{proof}
Fix $r > \varepsilon$. Let $\mathcal{L}$ be the set of projective lines meeting $\bar{B}_r$. We prove that $\mathcal{L}$ is a compact subset of $Gr_2(\R^{n+2})$. Let $\{P_i\}_i$ be a sequence of $\mathcal{L}$ converging to some $P$ in $Gr_2(\R^{n+2})$. For every $i$, there exist $q_i \in \pr(P_i) \cap \bar{B}_r$. Up to extracting, $\{q_i\}_i$ converges to some $q$ in $\bar{B}_r$. We have 
\[d(q, \pr(P)) \leq d(q, q_i) + d(q_i, \pr(P)) \leq d(q, q_i) + \delta(P_i, P).\]
By the convergence of $\{q_i\}_i$ to $q$ and the convergence of $\{P_i\}_i$ to $P$, we deduce that $q \in \pr(P) \cap \bar{B}_r$, i.e. $P \in \mathcal{L}$. Therefore, $\mathcal{L}$ is a close subset of $Gr_2(\R^{n+2})$ and then compact.
Now, we define the mapping $f : \pr(\R^{n+2}) \times \mathcal{L} \longrightarrow \R$ by $f(q,P) = d(q, \pr(P) \cap \bar{B}_r)/d(q,\pr(P))$. Notice that if $q \in B_{\varepsilon}$ and $q$ is "near" from $\pr(P)$ then $d(q, \pr(P)) = d(q,\pr(P) \cap \bar{B}_r)$. It follows that $f$ is well-defined. Since $f$ is continuous on a compact set, $f$ is upper bounded, i.e. there exists $M > 0$ such that
\[d(q, \pr(P) \cap \bar{B}_r) \leq M.d(q,\pr(P))\]
for every $q \in \pr(\R^{n+2})$ and every $P \in \mathcal{L}$. 

\end{proof}

\begin{proof}[Proof of Proposition \ref{techn. prop}]
Let $p \in \Lambda$. By definition of the limit set, $p$ is an attracting point of some $P_1$-divergent sequence $\{g_i\}_i$ of $\rho(\Gamma)$. Every element $g_i$ admits a Cartan decomposition $g_i = k_i.a_i.\ell_i$ with $k_i$, $\ell_i \in SO(2) \times SO(n)$ and $a_i$ the Cartan projection of $g_i$ (see Section \ref{1.3}). According to Corollary \ref{dyn. pr} (Section \ref{1.3}), up to extracting, we can assume that $\{k_i\}_i$ and $\{\ell_i\}_i$ converge uniformly on the projective space to some $k, \ell \in SO(2)\times SO(n)$ such that $p = k.p_0$ with $p_0 = \pr(e_1)$.\\
By \cite{Gueritaud2017Anosov}, there exist $0 < \varepsilon < 1$ and $i_0 \in \mathbb{N}$ such that for every $i \geq i_0$
\[\limsup_{j \rightarrow +\infty} d(p_0, (a_i^{-1}k_i^{-1}k_{i+j}a_i).p_0) \leq 1 - \varepsilon.\]
Since $p_0$ is a fixed point of $a_i$ and since the metric $d$ is unvariant under $\ell_i^{-1} \in O(n+2)$, we deduce that for every $i \geq i_0$
\begin{align} \label{ineq conc}
d(p_i , g_i^{-1}.p) \leq 1 - \varepsilon
\end{align}
with $p_i = l_i^{-1}.p_0$.\\ 
Let $i \geq i_0$. We consider the open subset $W_i$ of $Gr_2(\R^{n+2})$ consisting of projective lines less than $(1 - \frac{\varepsilon}{2})$ apart from $p_i$ : 
\[W_i = \{P \in Gr_2(\R^{n+2}):\ d(p_i, \pr(P)) < 1 - \frac{\varepsilon}{2}\}.\]
\textit{Claim 1.} For every $q \in \pr(\R^{n+2})$, every projective line $P$ going through $q$ belongs to $W_i$ if and only if $q$ lies in the ball of radius $(1 - \frac{\varepsilon}{2})$ centered in $p_i$; i.e.
\begin{align}
F_q \subset W_i \Leftrightarrow d(p_i,q) < 1 - \frac{\varepsilon}{2}.
\end{align}

Indeed, assume $F_q \subset W_i$. In particular, the projective line going through $q$ such that $q$ realizes the distance to $p_i$ belongs to $W_i$, i.e. $d(p_i,q) < 1 - \frac{\varepsilon}{2}$. Conversely, for every $P \in F_q$, we have $d(p_i,\pr(P)) \leq d(p_i,q) < 1 - \frac{\varepsilon}{2}$, i.e. $P \in W_i$.\\ \\
\textit{Claim 2.} For every $c > 1$, there exists $i_c \in \mathbb{N}$ such that for every $i \geq i_c$
\begin{align}
\delta(P, F_q) \geq c.\delta(g_i.P, g_i.F_q)
\end{align}
for every $P \in W_i$ and every $q \in \pr(\R^{n+2})$ with $F_q \subset W_i$.\\

Indeed, fix $c > 1$. For $i \geq i_0$, consider $P  \in W_i$ and $q \in \pr(\R^{n+2})$ with $F_q \subset W_i$, i.e. $d(p_i,q) < 1 - \frac{\varepsilon}{2}$ (see Claim $1$). Let $q_i := \ell_i.q$ and $P_i := \ell_i.P$. Since the metric $d$ in invariant under $\ell_i$, $q_i$ belongs to $B_{1 - \frac{\varepsilon}{2}}$. According to Fact \ref{techn. lemma 2}, there exist $0< r < \frac{\varepsilon}{2}$ and $M > 0$ such that
\begin{align} \label{ineq}
d(q_i, \pr(P_i) \cap \bar{B}_{1-r}) \leq M. d(q_i, \pr(P_i)).
\end{align}
Since $\{a_i\}_i$ is $P_1$-divergent, there exists an integer $i_{r,\frac{1}{cM}}$ such that for every $j \geq i_{r,\frac{1}{cM}}$, the map ${a_j}|_{\bar{B}_{1-r}}$ is $\frac{1}{cM}$-contracting.\\
Let $i_c : = \max (i_0, i_{r,\frac{1}{cM}})$ and $i \geq i_c$. We have
\begin{align*}
d(a_i.q_i, a_i.\pr(P_i)) &\leq d(a_j.q_i, a_j.(\pr(P_i) \cap \bar{B}_{1-r})) \\
                         &\leq \frac{1}{cM} d(q_i, \pr(P_i) \cap \bar{B}_{1-r})).
\end{align*}
According to the inequality \ref{ineq}, we deduce
\begin{align} \label{main ineq}
d(a_i.q_i, a_i.\pr(P_i)) &\leq \frac{1}{c} d(q_i, \pr(P_i)).
\end{align}
Since $d$ is invariant under $\ell_i$, we have $d(q_i, \pr(P_i)) = d(q, \pr(P))$ and according to Fact \ref{techn. lemma 1}, $d(q, \pr(P)) = \delta (P, F_q)$. Thus
\begin{align} \label{equ 1}
d(q_i, \pr(P_i)) &= \delta (P, F_q).
\end{align}
Now, since $d$ is invariant under $k_i$, we have $d(a_i.q_i, a_i.\pr(P_i)) = d(g_i.q, g_i.\pr(P))$ and according to Fact \ref{techn. lemma 1}, $d(g_i.q, g_i.\pr(P)) = \delta (g_i.P, g_i.F_q)$. Thus
\begin{align} \label{equ 2}
d(a_i.q_i, a_i.\pr(P_i)) &= \delta (g_i.P, g_i.F_q).
\end{align}
It follows from the inequality \ref{main ineq} and the equalities \ref{equ 1} and \ref{equ 2}
\[\delta (P, F_q) \geq c.\delta (g_i.P, g_i.F_q).\]
This proves Claim $2$.\\

To conclude, fix $i \geq i_c$. By Claim $1$ and the inequality \ref{ineq conc}, we have $F_{g_i^{-1}.p} \subset W_i$, hence $F_p \subset g_i.W_i$. Then Proposition \ref{techn. prop} follows from Claim $2$ by setting $g = g_i$ and $W = g_i.W_i$.
\end{proof}

\begin{proof}[Proof of the compactness in Theorem \ref{GKW}]
Let $F$ be the set of projective lines which meet the limit set $\Lambda$. Let us test Criterion \ref{compactness criterion} for $G = \rho(\G)$, $Z = Gr_2(\R^{n+2})$, $F$ and $K = \Lambda$. According to Proposition \ref{techn. prop}, the condition \ref{condition criterion} of Criterion \ref{compactness criterion} is satisfied. It follows that the action of $\rho(\Gamma)$ on the set $Z \backslash F$ of projective lines avoiding the limit set is cocompact.\\

Recall that $U$ is the set of causal geodesics which avoid the limit set. Let $\{\varphi_i\}_i$ be a sequence of $U$ and $\{\bar{\varphi}_i\}_i$ its image in $\rho(\G) \backslash U$ by the canonical projection. Every $\varphi_i$ is defined by a $2$-plane $P_i$ in $\bar{X} \backslash F$. By the compactness of $\rho(\G) \backslash (Z \backslash F)$, there exist a subsequence $\{P_{i_j}\}_j$ and a sequence $\{g_j\}_j$ of $\rho(\G)$ such that $\{g_j.P_{i_j}\}_j$ converges to some $P$ in $Z \backslash F$. Since every $\{g_i\}_i$ is an isometry, $\{g_j.P_{i_j}\}_j$ is contained in $\bar{X}$ which is compact. Thus, $P \in \bar{X}$. It follows that $P$ define a causal geodesic $\varphi$ in $U$. Therefore, $\{\bar{\varphi_{i_j}}\}_j$ converges in $\rho(\G) \backslash U$. This proves the compactness of $\rho(\G) \backslash U$.
\end{proof}

\section{Conformally flat spacetimes and Anosov representations}\label{4}
In this section, we prove our main theorem.

\begin{theorem}\label{main theorem}
Any $P_1$-Anosov representation $\rho$ of a hyperbolic group $\Gamma$ into $O_0(2,n)$ ($n \geq 2$) with negative limit set $\Lambda$\footnote{It means that $\Lambda$ lifts to an acausal subset of $\eu$.} is the holonomy group of a CGHM conformally flat spacetime $M_{\rho}(\Gamma)$. Moreover, when $\Lambda$ is not a topological $(n-1)$-sphere, $\rho$ is the holonomy group of an \emph{$AdS$-spacetime with black hole} \footnote{See Section \ref{4.3}, $\S$ \ref{4.3.2}.} which embeds conformally in $M_{\rho}(\Gamma)$. 
\end{theorem}

Let $\Gamma$ be a hyperbolic group and $\rho : \Gamma \longrightarrow O_0(2,n)$ be a $P_1$-Anosov representation with negative limit set $\Lambda \subset \mathtt{Ein}_{1,n-1}$. The proof of Theorem \ref{main theorem} involves \emph{the invisible domain} of $\Lambda$ in $\mathtt{Ein}_{1,n}$. In \cite{andersson2012},\cite{barbot2015deformations}, \cite{Merigot2007AnosovAR}, the authors define the invisible domain of an achronal subset of $Ein_{1,n-1}$ in $AdS_{1,n}$ also called \emph{$AdS$-regular domain}. By analogy, we define in Section \ref{4.1} the invisible domain of $\Lambda$ in $\mathtt{Ein}_{1,n}$ and describe its geometric and dynamical properties.

\subsection{The invisible domain of the limit set} \label{4.1}

Let $\tilde{\Lambda}$ be the acausal subset of $\eu$ which projects on $\Lambda$. Since $\eu \subset \eeu$, we can see $\tilde{\Lambda}$ as an acausal subset of $\eeu$. Let $\tilde{\pi} : \eeu \longrightarrow Ein_{1,n}$ be the universal covering of $Ein_{1,n}$ and $\pi_2 : Ein_{1,n} \longrightarrow \mathtt{Ein}_{1,n}$ the double covering of $\mathtt{Ein}_{1,n}$. Since $\tilde{\Lambda}$ is acausal in $\eeu$, it is contained in an affine domain (see Proposition \ref{acausal set affine domain}). According to Remark \ref{fundamental group}, it follows that $\pi_2 \circ \pi$ restricted to $\tilde{\Lambda}$ is one to one. Therefore, $\tilde{\Lambda}$, $\tilde{\pi}(\tilde{\Lambda})$ and $\Lambda$ are conformally equivalent. In particular, $\tilde{\Lambda}$ and $\tilde{\pi}(\tilde{\Lambda})$ are compact. From now on, we just write $\Lambda$ to denote one of these three copies.

The group $O(2,n)$ is embedded in $O(2,n+1)$ as follow:
\[A \in O(2,n) \mapsto \begin{pmatrix}
A & 0 \\ 
0 & 1
\end{pmatrix} \in O(2,n+1).\]
We still denote by $\rho(\Gamma)$ its image by this embedding in $O(2,n+1)$. According to Section \ref{1.4}, there exists an onto morphism $j : \tilde{O}(2,n+1) \longrightarrow O(2,n+1)$ with kernel generated by $\delta$ (see Section \ref{1.3}). Let us consider the subgroup $G$ of elements of $j^{-1}(\rho(\Gamma))$ that preserves $\Lambda$. The restriction of $j$ to $G$ is one-to-one. Indeed, any element in the kernel of $j$ that preserves $\Lambda$ is necessarily equal to the identity since $\Lambda$ is contained in an affine domain. Therefore, the groups $G$ and $\rho(\Gamma)$ are isomorphic.\\

\begin{definition}
The \emph{invisible domain of $\Lambda$ in $\eeu$} is the region $\tilde{\Omega}(\Lambda)$ in $\eeu$ consisting on points which are not causally related to any point in $\Lambda$ :
\[\tilde{\Omega}(\Lambda) := \eeu \backslash (J^+(\Lambda) \cup J^-(\Lambda)).\]
\end{definition}

According to Proposition \ref{future of a compact}, $\tilde{\Omega}(\Lambda)$ is an open subset of $\eeu$. Besides, since $\Lambda$ is $\rho(\Gamma)$-invariant, $\tilde{\Omega}(\Lambda)$ is also $\rho(\Gamma)$-invariant.\\

\paragraph{\textit{Description in the conformal model.}} \label{4.1.1} Recall that $\Lambda \subset \eeu$ is the graph of a $1$-contracting function $f: \Lambda_0 \to \R$ where $\Lambda_0$ is a closed subset of the sphere $\mathbb{S}^{n}$ equipped with the distance $d_0$ induced by the round metric. It follows from the description of the future and the past of a point in $\eeu$ that each point $(x,t) \in \tilde{\Omega}(\Lambda)$ satisfies the inequality $d_0(x,x_0) > |t - f(x_0)|$ for every $x_0 \in \Lambda_0$. Hence
\begin{align*} 
\sup_{x_0 \in \Lambda_0} \{f(x_0) - d_0(x,x_0)\} \leq t \leq \inf_{x_0 \in \Lambda_0} \{f(x_0) + d_0(x,x_0).\}
\end{align*}
Since $\Lambda_0$ is compact, these two last inequalities are strict. Let $f^+, f^-: \mathbb{S}^n \to \R$ be the functions defined by
\begin{align*}
f^+(x) &= \inf_{x_0 \in \Lambda_0} \{f(x_0) + d_0(x,x_0)\}\\
f^-(x) &= \sup_{x_0 \in \Lambda_0} \{f(x_0) - d_0(x,x_0)\}.
\end{align*}
It is clear that any point $(x,t) \in \eeu$ such that $f^-(x) < t < f^+(x)$ lies in $\tilde{\Omega}(\Lambda)$.
Therefore, 
\begin{align*}
\tilde{\Omega}(\Lambda) &= \{(x,t) \in \mathbb{S}^n \times \R: f^-(x) < t < f^+(x)\}.
\end{align*}
It is easy to check that $f^+$ and $f^-$ are $1$-Lipschitz extensions of $f$.

Recall that $\Lambda \subset \eu \subset \eeu$. It follows that $\Lambda_0$ is contained in an equator $\mathbb{S}^{n-1} \subset \mathbb{S}^n$ which split the sphere $\mathbb{S}^n$ into two hemispheres $\mathcal{H}_1$ and $\mathcal{H}_2$. The restrictions $f^\pm_i$ of $f^\pm$ to $\mathcal{H}_i$ define two conformally isometric $AdS$-regular domains
\begin{align*}
E_i(\Lambda) &= \{(x,t) \in \mathcal{H}_i \times \R:\ f^-_i(x) < t < f^+_i(x)\}
\end{align*}
as described in \cite{andersson2012},\cite{barbot2015deformations} and \cite{Merigot2007AnosovAR}.
The restrictions $g^\pm$ of $f^\pm$ to $\mathbb{S}^{n-1}$ define the invisible domain of $\Lambda$ in $\eu$. This last one can be thought as the conformal boundary of $E_i(\Lambda)$ \footnote{It is actually in the sense of \cite{Barbot2005CausalPO} (see Section $9$).}, so we denote it by $\partial E(\Lambda)$. Finally, $\tilde{\Omega}(\Lambda)$ is the disjoint union of the conformally equivalent $AdS$-regular domains $E_i(\Lambda)$ and of their conformal boundary $\partial E(\Lambda)$.

\begin{remark}
Since $\Lambda$ is acausal, the regular domains $E_i(\Lambda)$ are non-empty (see Lemma 3.6 in \cite{Merigot2007AnosovAR}). Thus, $\tilde{\Omega}(\Lambda)$ is non-empty. 
\end{remark}

\begin{remark}
The invisible domain $\tilde{\Omega}(\Lambda)$ is causally convex. Indeed, let $p$ and $q$ be two points in $\tilde{\Omega}(\Lambda)$ related by a future causal curve $c$. If there exist a point $r$ in $c$ and a point $\lambda$ in $\Lambda$ such that $r \in J^+(\Lambda)$, then $q \in J^+(\lambda)$. Contradiction. By Proposition \ref{causally convex is GH}, it follows that $\tilde{\Omega}(\Lambda)$ is globally hyperbolic; in particular, $E_i(\Lambda)$ are strongly causal. Similarly, when $\partial E(\Lambda)$ is non-empty, it is globally hyperbolic.
\end{remark}

\paragraph{\textit{Description in the Klein model.}}\label{4.1.2}

Let us denote by $\Omega(\Lambda)$ the projection of $\tilde{\Omega}(\Lambda)$ in $Ein_{1,n}$. Let $\pi : \R^{2,n+1} \backslash \{0\} \longrightarrow \mathbb{S}(\R^{2,n+1})$ be the radial projection. In this paragraph, we see $\Lambda$ as a subset of $Ein_{1,n}$. Let $\Lambda_0$ be the preimage of $\Lambda$ by $\pi$. The convex hull $\mathrm{Conv}(\Lambda_0)$ of $\Lambda_0$ is a convex cone of $\R^{2,n+1}$. Let us consider its dual
\[\mathrm{Conv}(\Lambda_0)^* : = \{u \in \R^{2,n+1}:\ \forall v \in \mathrm{Conv}(\Lambda_0):\ <u,v> <0\}.\]

\begin{lemma}\label{Omega}
The domain $\Omega(\Lambda)$ is the intersection of $\pi(\mathrm{Conv}(\Lambda_0)^*)$ with $Ein_{1,n}$.
\end{lemma}

\begin{proof}
Remark that 
\[\Omega(\Lambda) = \pi(\{u \in \R^{2,n+1} :\ \forall \lambda \in \Lambda_0\ <u,\lambda>\ <0\}) \cap Ein_{1,n}.\]
\end{proof}

\begin{remark}\label{convex regular domains}
The regular domains $E_i(\Lambda)$ are the intersection of $\pi(\mathrm{Conv}(\Lambda_0)^*)$ with a copy of $\mathbb{ADS}_{1,n}$. A nice corollary of this is that $E_i(\Lambda)$ are geodesically convex, i.e. any geodesic joining two points in $E_i(\Lambda)$ is contained is $E_i(\Lambda)$. In particular, the regular domains are connected.
\end{remark}

\begin{remark}\label{equivalence}
It follows from Lemma \ref{Omega} that $\Omega(\Lambda)$ is contained in an affine domain of $Ein_{1,n}$. Indeed, any two distinct points $\lambda$, $\lambda'$ in $\Lambda_0$ define the affine domain $U(\frac{1}{2}(\lambda + \lambda'))$ (see Definition~\ref{affine domain}) which contains $\Omega(\Lambda)$.
\end{remark}

\begin{lemma}\label{one-to-one}
The universal covering $\tilde{\pi}$ restricted to $\tilde{\Omega}(\Lambda)$ is one-to-one.
\end{lemma}

\begin{proof}
Let $p$ and $q$ be two points in $\tilde{\Omega}(\Lambda)$ such that $\tilde{\pi}(p) = \tilde{\pi}(q)$. There exists $k \in \mathbb{Z}$ such that $q = \delta^k(p)$. Without loss of generality, assume $q \in J^+(p)$. If $k$ is nonzero, there is a future lightlike geodesic joining $p$ and $q$ and containing $\sigma(p)$. Since $\tilde{\Omega}(\Lambda)$ is causally convex, $\sigma(p)$ belongs to $\tilde{\Omega}(\Lambda)$ and thus $\Omega(\Lambda)$ contains the two opposite points $\tilde{\pi}(p)$ and $\tilde{\pi}(\sigma(p))$. This contradicts the fact that $\Omega(\Lambda)$ is contained in an affine domain.
\end{proof}

\begin{remark}
It follows from Remark \ref{equivalence} and Lemma \ref{one-to-one} that the invisible domain $\tilde{\Omega}(\Lambda)$ is contained in an affine domain of $\eeu$ (see Definition~\ref{affine domain in the universal covering}).
\end{remark}

According to Remark \ref{equivalence} and Lemma \ref{one-to-one}, the domains $\tilde{\Omega}(\Lambda)$, $\Omega(\Lambda)$ and $\pi_2(\Omega(\Lambda))$ are conformally equivalent. In what follows, we denote by $\Omega(\Lambda)$ the invisible domain of the limit set $\Lambda$ whether it is seen in $\eeu$, $Ein_{1,n}$ or $\mathtt{Ein}_{1,n}$.\\

\paragraph{\textit{Dynamical properties.}}\label{4.1.3}

In this section, we study the action of $\rho(\Gamma)$ on $\Omega(\Lambda)$.

\begin{proposition}
The group $\rho(\Gamma)$ acts freely and properly on the invisible domain $\Omega(\Lambda)$ in $\mathtt{Ein}_{1,n}$.
\end{proposition}

\begin{proof}
First, let us prove that the action is free. Assume there exist $p$ in $\Omega(\Lambda)$ and a nontrivial element $g$ in $\rho(\Gamma)$ such that $g.p = p$. Up to extracting, we can assume that $\{g^i\}_i$ is a sequence of pairwise distinct elements of $\rho(\Gamma)$. Since the representation $\rho$ is $P_1$-divergent, it follows that $\{g^i\}_i$ is $P_1$-divergent. Let $p_+$ and $p_-$ be respectively the attracting and the repelling points in $\mathtt{Ein}_{1,n}$ of $\{g^i\}_i$ (see Proposition \ref{dyn. pr ein}). Since $p$ does not belong to the lightlike cone of $p_-$, we have $\lim g^i.p = p_+$. But, for every $n \in \mathbb{N}$, we have $g^i.p = p$. Therefore, $p = p_+ \in \Lambda$, contradiction.\\

Now, we prove that the action is proper by proving that there is no points dynamically related in $\Omega(\Lambda)$. Let $\{p_i\}_i$ be a sequence of elements of $\Omega(\Lambda)$ converging to a point $p \in \Omega(\Lambda)$ and let $\{g_i\}_i$ be a sequence of $\rho(\Gamma)$ going to infinity (i.e. leaving every compact subset of $\rho(\Gamma)$) such that $\{g_i.p_i\}_i$ converges to a point $q$ in $\mathtt{Ein}_{1,n}$. Up to extracting, we can assume that $\{g_i\}_i$ is a sequence of pairwise distinct elements and thus $P_1$-divergent. Let $p_+$ and $p_-$ be respectively the attracting and the repelling points in $\mathtt{Ein}_{1,n}$ of $\{g_i\}_i$. Since $p$ does not belong to the lightlike cone $\mathcal{L}(p_-)$, the complementary of $\mathcal{L}(p_-)$ in $\mathtt{Ein}_{1,n}$ contains $\{p_i\}_i$ except maybe a finite number of elements. It follows that $\{g_i.p_i\}$ converge to $p_+$. Thus, $q = p_+ \not \in \Omega(\Lambda)$.
\end{proof}

\begin{corollary}
The quotient space $M_{\rho}(\Gamma) = \rho(\Gamma) \backslash \Omega(\Lambda)$ is a conformally flat spacetime.
\end{corollary}

\subsection{ CGHM conformally flat spacetime} \label{4.2}

In this section, we prove that the spacetime $M_{\rho}(\Gamma) = \rho(\Gamma) \backslash \Omega(\Lambda)$ is CGHM.\\

\paragraph{\textit{Global hyperbolicity.}} The following lemma states that $M_{\rho}(\Gamma)$ is causal.

\begin{lemma}\label{lemma GH}
The orbit under the action of $\rho(\Gamma)$ of any point in the invisible domain $\Omega(\Lambda)$ in $\mathtt{Ein}_{1,n}$ is an acausal subset of $\Omega(\Lambda)$.
\end{lemma}

\begin{proof}
Assume that there exist $p$ in $\Omega(\Lambda)$ and a non-trivial element $g$ in $\rho(\Gamma)$ such that $p$ and $g.p$ are causally related; for instance $g.p \in J^+(p)$. Up to extracting, the sequence $\{g^i\}_i$ is $P_1$-divergent. Let $p_+$ and $p_-$ be respectively the attracting and the repelling points of $\{g^i\}_i$. Since $p$ does not belong to the lightlike cone of $p_-$, we have $\lim g^i.p = p_+$. But, for every $i \in \mathbb{N}$, the point $g^i.p$ lies in the future of $p$ which is a close subset of $\eeu$. Thus, $p_+ \in J^+(p)$, contradiction.
\end{proof}

\begin{proposition}
The spacetime $M_{\rho}(\Gamma) = \rho(\Gamma) \backslash \Omega(\Lambda)$ is globally hyperbolic.
\end{proposition}

\begin{proof}
It immediately follows from Lemma \ref{lemma GH} that there is no causal loop in $M(\Gamma)$.

Let $\mathsf{p}$, $\mathsf{q}$ be two points in $M_{\rho}(\Gamma)$ such that the intersection $\mathsf{D} := J^+(\mathsf{p}) \cap J^-(\mathsf{q})$ is non-empty. For every $g \in \rho(\Gamma)$, we denote by $D_g$ the subset $J^+(p) \cap J^-(g.q)$ of $\Omega(\Lambda)$ where $p$ and $q$ are representatives of $\mathsf{p}$ and $\mathsf{q}$. It is easy to see that $\mathsf{D}$ is the projection in $M_{\rho}(\Gamma)$ of $\underset{g \in \rho(\Gamma)}{\bigcup} D_g \subset \Omega(\Lambda)$. It turns out that there is only a finite number of $g$ in $\rho(\Gamma)$ for which $D_g$ is non-empty. Indeed, assume that there is an infinite sequence $\{g_i\}_i$ of pairwise distinct elements of $\rho(\Gamma)$ such that for every $i \in \mathbb{N}$, the subset $D_{g_i}$ in non-empty. Since the representation $\rho$ is $P_1$-divergent, the sequence $\{g_i\}_i$ is $P_1$-divergent. Let $p_+$ and $p_-$ be respectively the attracting and the repelling points of $\{g_i\}_i$ in $\Lambda$. Since $q$ does not belong to the lightlike cone of $p_-$, we have $\lim g_i.q = p_+$. But, for every $i$, the point $g_i$ lies in the future of $p$ which is a close subset of $\eeu$. Thus, $p_+ \in J^+(p)$, contradiction.
\end{proof}

\paragraph{\textit{Spatial cocompactness.}} Recall that the \emph{space of causal geodesics} $\mathcal{C}$ is the space containing timelike and lightlike geodesics of $AdS_{1,n}$ and lightlike geodesics of $Ein_{1,n-1}$.
Let $U$ be the subspace of $\mathcal{C}$ consisting of causal geodesics which avoid the limit set. We denote by $\mathcal{P}(\Omega(\Lambda))$ the space of lightlike geodesics of $\Omega(\Lambda)$.

\begin{lemma} \label{lemma C1}
Each causal geodesic $\varphi$ which avoid the limit set meets the invisible domain $\Omega(\Lambda)$. Besides, the intersection $\varphi \cap \Omega(\Lambda)$ is connected.
\end{lemma}

\begin{proof}
Assume $\varphi$ is contained in $J^{+}(\Lambda) \cup J^-(\Lambda)$. One can write $\varphi$ as the union of the two closed subsets $(\varphi \cap J^{+}(\Lambda))$ and $(\varphi \cap J^{-}(\Lambda))$. Since $\Lambda$ is acausal, $J^{+}(\Lambda) \cap J^{-}(\Lambda) = \Lambda$. Thus, the intersection $(\varphi \cap J^{+}(\Lambda)) \cap (\varphi \cap J^{-}(\Lambda))$ is empty. Since $\varphi$ is connected, it follows that either $\varphi \cap J^{+}(\Lambda)$ or $\varphi \cap J^{-}(\Lambda)$ is empty. Assume, for instance that $\varphi \cap J^{-}(\Lambda)$ is empty, i.e. $\varphi$ is contained in $J^+(\Lambda)$. According to Proposition \ref{future of a compact}, the projection of $J^+(\Lambda) \subset \mathbb{S}^n \times \R$ onto $\R$ is bounded below and then $J^{+}(\Lambda)$ cannot contain the inextendible curve $\varphi$, contradiction. The fact that $\varphi \cap \Omega(\Lambda)$ is connected comes immediately from the causal convexity of $\Omega(\Lambda)$.
\end{proof} 

\begin{corollary}\label{corollary}
There exists a canonical homeomorphism between $\mathcal{P}(\Omega(\Lambda))$ and $\partial \mathcal{C} \cap U$.
\end{corollary}

\begin{proof}
By definition, a lightlike geodesic of $\Omega(\Lambda)$ is a connected component of the intersection of a lightlike geodesic of $\eeu$ with $\Omega(\Lambda)$. Let $i : \mathcal{P}(\Omega(\Lambda)) \longrightarrow \mathcal{P}(\eeu)$ be the map which associates to a lightlike geodesic of $\Omega(\Lambda)$ the lightlike geodesic of $\eeu$ that it comes from. The map $i$ is clearly continous and open. The image of $i$ is the space of lightlike geodesics of $\eeu$ that meet $\Omega(\Lambda)$; according to Lemma \ref{lemma C1}, it is equal to $\partial \mathcal{C} \cap U$. Moreover, the pre-images by $i$ of a lightlike geodesic $\varphi$ in $\partial \mathcal{C} \cap U$ is the set consisting of connected components of $\varphi \cap \Omega(\Lambda)$. Since $\Omega(\Lambda)$ is causally convexe, the intersection $\varphi \cap \Omega(\Lambda)$ is connected and thus $i^{-1}(\varphi)$ is reduced to one element. Therefore, $i$ is a homeomophism between $\mathcal{P}(\Omega(\Lambda))$ and $\partial \mathcal{C} \cap U$. 
\end{proof}

It follows from Theorem \ref{GKW} that the action of $\rho(\Gamma)$ on $\partial \mathcal{C} \cap U$ is properly discontinuous. One can easily see that the homeomorphism $i$ defined in the proof of Corollary \ref{corollary} above induces a homeomorphism between $\rho(\Gamma) \backslash \mathcal{P}(\Omega(\Lambda))$ and $\rho(\Gamma) \backslash (\partial \mathcal{C} \cap U)$.

\begin{proposition}\label{compact}
The quotient space $\rho(\Gamma) \backslash \mathcal{P}(\Omega(\Lambda))$ is compact.
\end{proposition}

\begin{proof}
Let $\{\mathrm{x}_i\}_i$ be a sequence of $\rho(\Gamma) \backslash(\partial \mathcal{C} \cap U) \subset \rho(\Gamma) \backslash U$. According to Theorem \ref{GKW}, the space $\rho(\Gamma) \backslash U$ is compact. Therefore, we can assume, up to extracting, that $\{\mathrm{x}_i\}_i$ converge to $\mathrm{x} \in \rho(\Gamma) \backslash U$. Let $\{x_i\}_i$ be a sequence in $\partial \mathcal{C} \cap U$ and $x$ a point in $U$ such that for every $i \in \mathbb{N}$, the point $x_i$ is representative of $\mathrm{x}_i$ and $x$ is representative of $\mathrm{x}$. There exists a sequence $\{g_i\}_i$ in $\rho(\Gamma)$ such that $\{g_i.x_i\}_i$ converges to $x$. Since $\rho(\Gamma)$ acts on $U$ by isometries, $\{g_i.x_i\}_n$ is a sequence of $\partial \mathcal{C} \cap U$. But $\partial \mathcal{C}$ is a close subset of~$\mathcal{C}$. Hence, $x \in \partial \mathcal{C} \cap U$.
\end{proof}

\begin{lemma}\label{homeo}
There exists a canonical homeomorphism between the spaces $\rho(\Gamma) \backslash \mathcal{P}(\Omega(\Lambda))$ and $\mathcal{P}(\rho(\Gamma) \backslash \Omega(\Lambda))$.
\end{lemma}

\begin{proof}
The canonical projection of $\Omega(\Lambda)$ onto $\rho(\Gamma) \backslash \Omega(\Lambda)$ induces a continuous surjective map from $\mathcal{P}(\Omega(\Lambda))$ to $\mathcal{P}(\rho(\Gamma) \backslash \Omega(\Lambda))$. This map descends to the quotient into a one-to-one continuous map $f$ from $\rho(\Gamma) \backslash \mathcal{P}(\Omega(\Lambda))$ to $\mathcal{P}(\rho(\Gamma) \backslash \Omega(\Lambda))$. Since $\rho(\Gamma) \backslash \mathcal{P}(\Omega(\Lambda))$ is compact (see Proposition \ref{compact}), $f$ is a homeomorphism.
\end{proof}

\begin{proposition}\label{unit tangent bundle}
Let $M$ be a GH spacetime and $S$ a Cauchy hypersurface of $M$. Let $\mathcal{P}(M)$ be the space of lightlike geodesics of $M$. There exists a canonical homeomorphism between $\mathcal{P}(M)$ and the unit tangent bundle $T^1S$ of $S$.
\end{proposition}

\begin{proof}
Recall that $M$ is homeomorphic to $S \times \R$ (see \cite{Choquet-Bruhat}). We fix a copy $S \times \{0\}$ of $S$ in $M$. Notice that for every $p=(x,t) \in M$, $T_pM = T_xS \oplus \R$. Let $u$ be a unit vector field on $\R$. Let $\varphi \in \mathcal{P}(M)$. The lightlike geodesic $\varphi$ of $M$ meets $S \times \{0\}$ in a unique point $p= (x,0)$. Let $w_{\varphi}$ be the tangent vector to $\varphi$ at $(x,0)$ such that $w_{\varphi} = v_{\varphi} + u(0)$ with $v_{\varphi} \in T_xS$. We define the map $f$ from $\mathcal{P}(M)$ to $T^1S$ which associates to $\varphi$ the point $(x,v_{\varphi}) \in T^1S$. Conversely, we define the map $g$ from $T^1S$ to $\mathcal{P}(M)$ which associates to $(x,v)$ the unique lightlike geodesic starting at $p = (x,0)$ with tangent vector $w = v + u(0)$. Clearly, $f$ and $g$ are continuous and $g$ is the inverse of $f$.
\end{proof}

\begin{proposition}
The conformally flat GH spacetime $M_{\rho}(\Gamma) = \rho(\Gamma) \backslash \Omega(\Lambda)$ is Cauchy compact. 
\end{proposition}

\begin{proof}
Let $S$ be a Cauchy hypersurface of $M_{\rho}(\Gamma)$. According to Proposition \ref{unit tangent bundle}, the spaces $\mathcal{P}(M(\Gamma))$ and $T^1S$ are homeomorphic. It follows from Proposition \ref{compact} and Lemma~\ref{homeo} that $T^1S$ is compact and thus $S$ is compact.
\end{proof}

\paragraph{\textit{Maximality.}} Let $\tilde{M}_{\rho}(\Gamma)$ be a universal covering of $M_{\rho}(\Gamma)$ and $D : \tilde{M}_{\rho}(\Gamma) \longrightarrow \Omega(\Lambda)$ be a developing map. Clearly, the holonomy group of $M_{\rho}(\Gamma)$ is $\rho(\Gamma)$.

\begin{proposition}
The CGH conformally flat spacetime $M_{\rho}(\Gamma) = \rho(\Gamma) \backslash \Omega(\Lambda)$ is maximal.
\end{proposition}

\begin{proof}
A globally hyperbolic spacetime is maximal if and only if its universal covering is maximal (see \cite{salveminithesis}, Chap.$7$, Corollary $2.7$). We prove that $\tilde{M}_{\rho}(\Gamma)$ is maximal. Let $N$ be a CGH  conformally flat spacetime and $f : \tilde{M}_{\rho}(\Gamma) \longrightarrow N$ a Cauchy embedding. Let $D' : N \longrightarrow \eeu$ be a developing map such that $D' \circ f = D$. Assume $f$ is not onto. Let $y \in \partial f(\Omega(\Lambda))$ and let $U$ be an open neighborhood of $y$ in $N$ such that the restriction of $D'$ to $U$ is one-to-one. Let $S$ be a Cauchy hypersurface of $\Omega(\Lambda)$. The union $\Omega' : = \Omega(\Lambda) \cup D'(U)$ is a globally hyperbolic spacetime and $S$ is a Cauchy hypersurface of $\Omega'$. Let $x \in \partial \Omega(\Lambda) \cap D'(U)$. Thus, there exists $\lambda$ in $\Lambda$ and a lightlike geodesic $\varphi$ joining $x$ and $\lambda$. It follows that $\varphi$ does not meet $S \subset \Omega(\Lambda)$, contradiction. 
\end{proof}

\subsection{AdS-spacetimes and black holes}\label{4.3}

In this section, we prove that the spacetime $M_{\rho}(\Gamma)= \rho(\Gamma) \backslash \Omega(\Lambda)$ is the union of two conformally equivalent $AdS$-spacetimes with black hole (see $\S \ref{4.3.2}$) except when the acausal limit set $\Lambda$ is a topological $(n-1)$-sphere. Recall that the invisible domain $\Omega(\Lambda)$ is the disjoint union of two $AdS$-regular domains $E_i(\Lambda)$ and of their conformal boundary $\partial E(\Lambda)$ (see $\S$ \ref{4.1.1}). The proof is based on the description of the regular domains $E_i(\Lambda)$. The case where $\Lambda$ is a topological $(n-1)$-sphere is already studied in \cite{andersson2012},\cite{barbot2015deformations} and \cite{Merigot2007AnosovAR}. In paragraph \ref{4.3.1}, we recall briefly the description in this particular case before dealing with the case where $\Lambda$ is not a topological $(n-1)$-sphere in paragraph \ref{4.3.2}.\\

\paragraph{\textit{Globally hyperbolic AdS-spacetimes.}} \label{4.3.1}

Assume $\Lambda$ is a topological $(n-1)$-sphere.

\begin{proposition}[Proposition $4.5$ and Theorem $4.3$ in \cite{Merigot2007AnosovAR}]\label{GH AdS regular domain}
The $AdS$-regular domains $E_i(\Lambda)$ are globally hyperbolic.
\end{proposition}

\begin{proposition}[Corollary 3.7 in \cite{Merigot2007AnosovAR}]\label{boundary AdS regular domain}
The conformal boundary $\partial E(\Lambda)$ of $E_i(\Lambda)$ is empty. Besides, the intersection of $\eu$ and the closure of $E_i(\Lambda)$ in $\eeu$ is reduced to $\Lambda$.
\end{proposition}

It follows from propositions \ref{GH AdS regular domain} and \ref{boundary AdS regular domain} that $\Omega(\Lambda)$ is the disjoint union of the two conformally equivalent globally hyperbolic regular $AdS$-domains $E_1(\Lambda)$ and $E_2(\Lambda)$.

\begin{proposition}[Theorem $4.3$ and Proposition $4.4$ in \cite{Merigot2007AnosovAR} ]
The quotients $M^i_{\rho}(\Gamma) = \rho(\Gamma) \backslash E_i(\Lambda)$ are globally hyperbolic $AdS$-spacetimes.
\end{proposition}

The spacetimes $M^i_{\rho}(\Gamma)$ embeds conformally in $M_{\rho}(\Gamma)$. Therefore, $M_{\rho}(\Gamma)$ is the disjoint union of conformal copies of two conformally equivalent globally hyperbolic $AdS$-spacetimes.\\

Let us give an exemple. Consider a cocompact torsion-free lattice $\Gamma$ in $O_0(1,n)$ and $\rho: \Gamma \to O_0(2,n)$ a Fuchsian representation i.e. the composition of the natural embedding $O_0(1,n) \subset O_0(2,n)$ and a faithful and dicrete representation of $\Gamma$ into $O_0(1,n)$. The limit set $\Lambda$ is a conformal $(n-1)$-sphere in $\eu$. Up to conformal isometry, we can assume that $\Lambda = \mathbb{S}^{n-1} \times \{0\}$. Let $x_i \in \mathcal{H}^n_i$ be the poles of $\mathbb{S}^n$.

\begin{proposition}
The invisible domain $\Omega(\Lambda)$ is the disjoint union of two diamonds $\Delta_i = I^-(p_i^+) \cap I^+(p_i^-)$ ($i=1,2$) with $p_i^\pm = (x_i,\pm \frac{\pi}{2})$.
\end{proposition}

\begin{proof}
The limit set is the graph of the null function $f$ on $\mathbb{S}^{n-1}$. Then, $F^\pm = d_0(.,\mathbb{S}^{n-1})$ and $\Omega(\Lambda) = \{(x,t) \in \mathbb{S}^n \times \R:\ |t| < d_0(x,\mathbb{S}^{n-1})\}$. It is clear that $\Omega(\Lambda)$ does not meet $\eu$. It follows that $\Omega(\Lambda)$ is the disjoint union of $\Omega(\Lambda) \cap (\mathcal{H}_i^n \times \R)$ $(i = 1,2)$. The distance from a point $x \in \mathcal{H}_i^n$ to $\mathbb{S}^{n-1}$ is equal to $d_0(x_i,\mathbb{S}^{n-1}) - d_0(x,x_i) = \pi/2 - d_0(x,x_i)$. Therefore, $\Omega(\Lambda) \cap (\mathcal{H}_i^n \times \R)$ is the set of points which satisfy $d(x,x_i) < t+\pi/2$ and $d(x,x_i) < \pi/2 - t$. Hence
\begin{align*}
\Omega(\Lambda) \cap (\mathcal{H}_i^n \times \R)&= (I^+(p_i^-) \cap I^-(p_i^+))
\end{align*}
with $p_i^\pm = (x_i,\pm \frac{\pi}{2})$.
\end{proof}

\begin{remark}
The diamonds $\Delta_i$ are diffeomorphic to $\R^{1,n}$.
\end{remark}

\paragraph{\textit{AdS-spacetimes with black holes.}} \label{4.3.2} Assume $\Lambda$ is not a topological $(n-1)$-sphere. We prove that the $AdS$-spacetimes $M^i_{\rho}(\Gamma) = \rho(\Gamma) \backslash E_i(\Lambda)$ contain a region which can be interpreted as a black hole. Before that, let us briefly recall some elements on black holes.\\

In 1992, the theoretical physicists M. Ba{\~n}ados, C. Teitelboim and J. Zanelli discovered that the standard Einstein-Maxwell equations in $2+1$ spacetime dimension, with negative cosmological constant admit a black hole solution (see \cite{Baados1992BlackHI}). This solution is named, after the authors, \emph{BTZ-black hole}. This came as a surprise since when the cosmological constant is zero, a vacuum solution of the Einstein-Maxwell equations in $2+1$ spacetime dimension is necessarily flat \footnote{In dimension $2+1$, the Einstein equation is remarkably simplified: the solutions have all constant sectional curvature with the same sign than the cosmological constant has.} and it has been shown that no black hole solutions with event horizons exist (see \cite{Baados1992BlackHI}). But, thanks to the negative cosmological constant, BTZ-black holes presents similar properties with the $3+1$ dimensional Schwarzschild and Kerr black hole solutions, which model real world black holes. For a quick insight into Schwarzschild and Kerr black holes, we direct the reader to the introduction of \cite{Barbot2005CausalPO}. A discussion around what could be a relevant mathematical definition of BTZ-black holes is also presented. Since the cosmological constant is negative, BTZ-black holes are $2+1$ dimensional $AdS$-spacetimes. Moreover, according to \cite{Barbot2005CausalPO}, the relevant category of spacetimes likely to be BTZ-black holes is the category of \emph{strongly causal} spacetimes. The following definition given in \cite{Barbot2005CausalPO}, involves the notion of \emph{conformal boundary} of $AdS$-spacetimes which is developed in Section $9$ of \cite{Barbot2005CausalPO}. 

\begin{definition}\label{BTZ-black hole}
A BTZ-black hole is a $2+1$ dimensional strongly causal $AdS$-spacetime $M$ such that
\begin{itemize}
\item $M$ admits a non-empty strongly causal conformal boundary $O$;\\
\item the \emph{past of $O$}, i.e. the region of $M$ made of initial points of future-oriented causal curves ending in $O$, is not the entire $M$.
\end{itemize} 
\end{definition}

The region $O$ is interpreted as the region where the observers take place. Every connected component $B_i$ of the interior of the complement in $M$ of the past of $O$ is a region invisible from $O$: no future causal curve, in particular no photon, can escape from it. In a more physical langage, no light and no information can escape from it. Hence, $B_i$ is interpreted as a black hole. In what follows, we observe a similar phenomenon in the regular domains $E_i(\Lambda)$ and in the $AdS$ spacetimes $M^i_{\rho}(\Gamma) = \rho(\Gamma) \backslash E_i(\Lambda)$, that we interpret by analogy as a black hole. In dimension $2+1$, the $AdS$-spacetimes $M^i_{\rho}(\Gamma)$ correspond exactly to BTZ-black holes.\\

In all what follows, $\Lambda$ is not a topological $(n-1)$-sphere. We describe the regular domains $E_i(\Lambda)$. Recall that there are two maps $g^\pm: \mathbb{S}^{n-1} \subset \mathbb{S}^n \to \R$ such that $\partial E(\Lambda) = \{(x,t);\ g^-(x) < t < g^+(x)\}$ (see $\S \ref{4.1.1}$) and that we denote by $\mathcal{H}_1, \mathcal{H}_2$ the hemispheres of $\mathbb{S}^n$ bounded by $\mathbb{S}^{n-1}$. Let $\Lambda^\pm$ be the graphs of $g^\pm$. They are achronal $(n-1)$-spheres of $\eu$ which contain $\Lambda$. It follows that the $AdS$-regular domains $E_i(\Lambda^\pm)$ in $\mathcal{H}_i^n \times \R$ are maximal globally hyperbolic open domains of $E_i(\Lambda)$ (see \cite{Merigot2007AnosovAR}). This fact justifies the following definition used in \cite{Barbot2005CausalPO}.

\begin{definition}
The regular domain $E_i(\Lambda^+)$ (resp. $E_i(\Lambda^-)$) is the \emph{future (resp. past) globally hyperbolic convex core of $E_i(\Lambda)$}.
\end{definition}

Here is another definition, introduced in \cite{Barbot2005CausalPO}, useful in the description of the $AdS$-regular domains $E_i(\Lambda)$.

\begin{definition}
The boundary of $E_i(\Lambda^+)$ (resp. $E_i(\Lambda^-)$) in $E_i(\Lambda)$ is called the \emph{future (resp. past) horizon}.
\end{definition}

Now, we can state our description of $E_i(\Lambda)$ which generalizes, to the higher dimension, the description in dimension $2+1$ given in \cite{Barbot2005CausalPO}, Section $3.1$.

\begin{proposition}\label{black hole in AdS-domain}
The $AdS$-regular domain $E_i(\Lambda)$ is the disjoint union of the future globally hyperbolic convex core $E_i(\Lambda^+)$, the past of $\partial E(\Lambda)$ in $E_i(\Lambda)$ and the future horizon.
\end{proposition}

It is easy to see that the globally hyperbolic domain $\partial E(\Lambda)$ is the conformal boundary of $E_i(\Lambda)$ in the sense of Definition $9.3$ in \cite{Barbot2005CausalPO}. The globally hyperbolic convex core $E_i(\Lambda^+)$ is the region of $E_i(\Lambda)$ invisible from $\partial E(\Lambda)$: no future causal curve starting from a point of $E_i(\Lambda^+)$ reach a point in $\partial E(\Lambda)$. In a more "physical" language, photons can't escape from $E_i(\Lambda^+)$ which is consequently not visible to the obervers in $\partial E(\Lambda)$. The past of $\partial E(\Lambda)$ in $E_i(\Lambda)$ is the region visible from $\partial E(\Lambda)$. In the convention of \cite{Baados1992BlackHI} and \cite{Baados1993GeometryOT}, $E_i(\Lambda^+)$ and the past of $\partial E(\Lambda)$ are respectively called \emph{the intermediate} and \emph{the outer region}. These two regions are separated by \emph{the future horizon}.\\

To prove Proposition \ref{black hole in AdS-domain}, we need the following lemmas.

\begin{lemma}\label{B1}
The conformal boundary $\partial E(\Lambda)$ is the intersection of the past of $\Lambda^+ \backslash \Lambda$ and the future of $\Lambda^- \backslash \Lambda$ in $\eu$.
\end{lemma}

\begin{proof}
It is easy to see that a point $(x,t) \in \partial E(\Lambda)$ belongs to $I^-(p^+) \cap I^+(p^-)$ with $p^\pm =~(x,f^\pm(x)) \in \Lambda^\pm \backslash \Lambda$. Conversely, let $(x,t)\in I^-(p_0) \cap I^+(p_1)$ with $p_0 = (x_0,f^+(x_0))$ and $p_1 = (x_1,f^-(x_1))$ where $x_0,x_1 \in \mathbb{S}^{n-1} \backslash \Lambda_0$. Thus, $d_0(x,x_0) < f^+(x_0) - t$ and $d_0(x,x_1) < t - f^-(x_1)$. Since $f^+$ is $1$-contracting, $f^+(x_0) - f^+(x) < d_0(x,x_0) < f^+(x_0) -~t$. Hence $t < f^+(x)$. Similarly, using the fact that $f^-$ is $1$-contracting, we obtain $f^-(x) < t$. It follows that $f^-(x) < t < f^+(x)$ i.e. $(x,t) \in \partial E(\Lambda)$.
\end{proof}

\begin{lemma}\label{B2}
The past of $\partial E(\Lambda)$ in $E_i(\Lambda)$ coincide with the past of $\Lambda^+ \backslash \Lambda$ in $E_i(\Lambda)$.
\end{lemma}

\begin{proof}
By Lemma \ref{B1}, the past of $\partial E(\Lambda)$ is contained in the past of $\Lambda^+ \backslash \Lambda$. Conversely, let $p \in I^-(p_0)$ with $p_0 \in \Lambda^+ \backslash \Lambda$. Then, $I^+(p)$ is an open neighborhood of $p_0$ in $\eeu$. Notice that $p_0$ belongs to the boundary of $\partial E(\Lambda)$. Therefore, $I^+(p)$ meets $\partial E(\Lambda)$. The lemma follows.
\end{proof}

\begin{proof}[Proof of Proposition \ref{black hole in AdS-domain}]
In the Klein model, we have the following description of $E_i(\Lambda)$
\[E(\Lambda) = \{x \in \mathbb{ADS}_{1,n}:\ <x,y> < 0\ \forall y \in \Lambda\}.\]
Let $x \in E(\Lambda)$. Either $<x,y> < 0$ for any $y \in \Lambda^+$ or $<x,y> \geq 0$ for some $y \in \Lambda^+$. In other words, $x$ belongs either to $E(\Lambda^+)$ or to the closure of the past of $\Lambda^+\backslash \Lambda$. The proposition follows from Lemma \ref{B2}.
\end{proof}

\begin{remark}
Similarly, $E_i(\Lambda)$ is the disjoint union of the past globally hyperbolic convex core, the future of $\partial E(\Lambda)$ in $E_i(\Lambda)$ and the past horizon. The past globally hyperbolic convex core $E(\Lambda^-)$ is the region of $E_i(\Lambda)$ which cannot be entered from the outside: no past causal curve can escape from it.
\end{remark}

Now, consider the quotient spacetimes
\[M^i_{\rho}(\Gamma) = \rho(\Gamma) \backslash E_i(\Lambda),\  B^i_{\rho}(\Gamma) = \rho(\Gamma) \backslash E_i(\Lambda^+),\ O_{\rho}(\Gamma) = \rho(\Gamma) \backslash \partial E(\Lambda).\]

The conformal embeddings $E_i(\Lambda) \subset \Omega(\Lambda)$ induces conformal embeddings from the $AdS$-spacetimes $M^i_{\rho}(\Gamma)$ into $M_{\rho}(\Gamma)$. Since $M_{\rho}(\Gamma)$ is globally hyperbolic, $M^i_{\rho}(\Gamma)$ are strongly causal. 

\begin{proposition}[Theorem $4.3$ and Proposition $4.4$ in \cite{Merigot2007AnosovAR}]
The $AdS$-spacetimes $B^i_{\rho}(\Gamma)$ are globally hyperbolic.
\end{proposition}

\begin{remark}
According to remark \ref{convex regular domains}, the regular domains $E_i(\Lambda)$ and $E_i(\Lambda^+)$ are connected. Therefore, the $AdS$-spacetimes $M^i_{\rho}(\Gamma)$ and $B^i_{\rho}(\Gamma)$ are also connected.
\end{remark}

\begin{proposition}
The conformally flat spacetime $O_{\rho}(\Gamma)$ is CGHM.
\end{proposition}

\begin{proof}
The proof presented in Section \ref{4.2} still holds for the non-empty invisible domain $\partial E(\Lambda)$ of the limit set in $\eu$.
\end{proof}

The spacetimes $M^i_{\rho}(\Gamma)$, $B^i_{\rho}(\Gamma)$ and $O_{\rho}(\Gamma)$ embeds conformally in $M_{\rho}(\Gamma)$. We keep the same notation to denote their image in $M_{\rho}(\Gamma)$.\\

\begin{proposition}\label{black hole}
The globally hyperbolic $AdS$-spacetime $B^i_{\rho}(\Gamma)$ is the region of $M^i_{\rho}(\Gamma)$ invisible from $O_{\rho}(\Gamma)$, i.e. no future causal curve starting from a point of $B^i_{\rho}(\Gamma)$ can reach $O_{\rho}(\Gamma)$. 
\end{proposition}

\begin{proof}
This is an immediate consequence of Proposition \ref{black hole in AdS-domain}.
\end{proof}

Since $\partial E(\Lambda)$ is the conformal boundary of $E_i(\Lambda)$, it easily follows that the GHCM spacetime $O_{\rho}(\Lambda)$ is the conformal boundary of $M^i_{\rho}(\Lambda)$. According to Proposition \ref{black hole} and Definition \ref{BTZ-black hole}, in dimension $2+1$, the strongly causal $AdS$-spacetimes $M^i_{\rho}(\Gamma)$ are BTZ-black holes. More precisely, the globally hyperbolic region $B^i_{\rho}(\Lambda)$ is interpreted as a black hole in $M^i_{\rho}(\Lambda)$. In higher dimension, we still interpret the region $B^i_{\rho}(\Gamma)$ as a black hole in $M^i_{\rho}(\Gamma)$. In this sense, $M_{\rho}(\Gamma)$ is the union of two conformally equivalent $AdS$-spacetimes with black holes.\\

Let us give some examples. Let $\Gamma$ be a uniform lattice in $O_0(1,p+1)$ ($0 \leq p \leq n-2$) and $\rho: \Gamma \to O_0(2,n)$ be the composition of the natural inclusions $\Gamma \subset O_0(1,p+1)$ and $O_0(1,p+1) \subset O_0(2,n)$ where in the latter inclusion, $O_0(1,p+1)$ is the stabilizer of a totally geodesic spacelike subspace $\hy^{p+1} = \R^{1,p+1} \cap AdS_{1,n}$ in $AdS_{1,n}$. The limit set is the conformal sphere $\mathbb{S}^p = \partial \hy^{p+1}$. The group $\rho(\Gamma)$ preserves the orthogonal of $\R^{1,p+1}$ in $\R^{2,n}$ which is isometric to $\R^{1,q+1}$ with $q = n-p-2$. Let $\mathbb{S}^q$ be the conformal sphere of $\R^{1,q+1}$. In what follows, we describe the intermediate region $E_i(\Lambda^+)$. This description has great analogies with that of \emph{split AdS spacetimes} introduced in \cite{barbot2015deformations}.

\begin{proposition}
The topological $(n-1)$-sphere $\Lambda^+$ is the join of $\mathbb{S}^p$ and $\mathbb{S}^q$ i.e. the union of $\mathbb{S}^p$, $\mathbb{S}^q$ and the lightlike geodesics joining a point of $\mathbb{S}^p$ to a point of $\mathbb{S}^q$.
\end{proposition}

\begin{proof}
Up to a conformal isometry, we can assume that $\Lambda \subset \mathbb{S}^{n-1} \times \{0\}$. Then, the function $f^+$ is the distance to $\mathbb{S}^p$: $f^+(x) = d_0(x,\mathbb{S}^p)$ for $x \in \mathbb{S}^{n-1}$. It is easy to see that any point in $\Lambda^+$ lies in the future lightcone of some point in $\Lambda$. Indeed, for any $x \in \mathbb{S}^{n-1}$ there exists $x_0 \in \mathbb{S}^p$ such that $d_0(x,\mathbb{S}^p) = d_0(x,x_0)$. In other words, the point $(x,d_0(x,\mathbb{S}^p)) \in \Lambda^+$ belongs to the future lightcone of $(x_0,0) \in \Lambda$. Notice that the lightcones of points in $\Lambda$ intersect at $\mathbb{S}^q$. Indeed, in the Klein model, a point $[v]$ of $Ein_{1,n-1}$ in this intersection satisfies $<v,v_0> = 0$ for any lightlike vector $v_0 \in \R^{1,p+1}$. Since $\R^{1,p+1}$ admits a basis composed of lightlike vectors, it follows that $v$ is orthogonal to $\R^{1,p+1}$. Thus, $[v] \in \mathbb{S}^q$. In the conformal model, $\mathbb{S}^q$ is contained in $\mathbb{S}^{n-1} \times \{\pi/2\}$. It is clear that $\mathbb{S}^q$ is contained in $\Lambda^+$. Besides, if $(x,t)$ belongs to a lightlike geodesic joining a point $(x_0,0)$ in $\Lambda$ to a point $(x_1,\pi/2)$ in $\mathbb{S}^q$ then $x$ belongs to the geodesic of $\mathbb{S}^{n-1}$, parametrized by its lenght arc, which realises the distance of $x_1$ to $\mathbb{S}^p$. Thus, $t = d_0(x,x_0) = d_0(x,\mathbb{S}^p)$. In other words, $(x,t) \in \Lambda^+$.
\end{proof}

\begin{proposition}[Lemma $4.27$ in \cite{barbot2015deformations}]
The invisible domain $E(\Lambda^+)$ is the interior of the convex hull in $AdS_{1,n}$ of $\mathbb{S}^p \cup \mathbb{S}^q$.
\end{proposition}

\begin{remark}
Cauchy hypersurfaces of $E(\Lambda^+)$ are homeomorphic to $\hy^{p+1} \times \hy^{q+1}$.
\end{remark}

\bibliography{Biblio}
\bibliographystyle{plain}

\end{document}